\documentclass[reqno]{amsart}
\usepackage{cmll}

\usepackage{graphicx} % Required for inserting images
\usepackage{mathrsfs}
\usepackage{amssymb}
\usepackage{xcolor}
\usepackage{enumitem}
\usepackage{proof}
\usepackage{hyperref}
\hypersetup{
    colorlinks,
    linkcolor={black},
    urlcolor={black}
}

\newcommand{\rn}[1]{\mathsf{#1}}
\newcommand{\efq}{\textsf{EFQ}}
\newcommand{\dne}{\textsf{DNE}}
\newcommand{\gen}{\textsf{GEN}}
\renewcommand{\mp}{\textsf{MP}}

\title[Proof-theoretic Semantics for First-order Logic]{Proof-theoretic Semantics for \\ First-order Logic}
\author{Alexander V. Gheorghiu}
\address{University of Southampton \& University College London}
\email{a.v.gheorghiu@soton.ac.uk}

\newcommand{\uniclose}[1]{\forall \vec{x}#1}
\newcommand{\supp}{\Vdash}
\newcommand{\proves}{\vdash}
\newcommand{\base}[1]{\mathscr{#1}}
\newcommand{\baseB}{\base{B}}
\newcommand{\baseX}{\base{X}}

\newcommand{\setClassical}{\mathsf{C}}
\newcommand{\setIntuitionistic}{\mathsf{I}}
\newcommand{\FV}{\mathsf{FV}}
\newcommand{\powerset}{\mathcal{P}}
\newcommand{\seq}{\triangleright}
\newcommand{\kflat}[1]{#1^\flat}
\newcommand{\ksharpl}[1]{#1^{\natural_1}}
\newcommand{\ksharpr}[1]{#1^{\natural_2}}
\newcommand{\atset}[1]{\mathbb{#1}}
\newcommand{\setEigen}{\mathcal{E}}
\newcommand{\basesup}[1]{\supseteq_{#1}}
\newcommand{\basis}[1]{\mathfrak{#1}}

\makeatletter
\def\labelandtag#1#2{\begingroup
   \def\@currentlabel{#2}%
   \phantomsection\label{#1}\endgroup
}
\makeatother

\newcommand{\setVar}{\mathcal{V}}
\newcommand{\setCons}{\mathcal{K}}
\newcommand{\setFunc}{\mathcal{F}}
\newcommand{\setPred}{\mathcal{P}}
\newcommand{\setTerm}{\mathcal{T}}
\newcommand{\closure}{\textsc{cl}}
\newcommand{\setClosedTerm}{\closure(\setTerm)}
\newcommand{\setClosedAtom}
{\closure(\setAtoms)}
\newcommand{\setClosedFormulas}{\closure(\setFormulas)}
\newcommand{\setFormulas}{\mathcal{F}}
\newcommand{\setAtoms}{\mathcal{A}}

\newcommand{\irn}[1]{\rn{#1}\textsf{I}}
\newcommand{\ern}[1]{\rn{#1}\textsf{E}}

\newtheorem{definition}{Definition}
\newtheorem{theorem}[definition]{Theorem}
\newtheorem{example}[definition]{Example}
\newtheorem{proposition}[definition]{Proposition}
\newtheorem{lemma}[definition]{Lemma}

\begin{document}

\begin{abstract}
    Sandqvist gave a \emph{proof-theoretic semantics} (P-tS) for classical logic (CL) that explicates the meaning of the connectives without assuming bivalance. Later, he gave a semantics for intuitionistic propositional logic (IPL). While soundness in both cases is proved through standard techniques, the proof completeness for CL is complex and somewhat obscure, but clear and simple for IPL. Makinson gave a simplified proof of the completeness of classical propositional logic (CPL) by directly relating the P-tS to the logic's extant truth-functional semantics. In this paper, we give an elementary, constructive, and native --- in the sense that it does not presuppose the model-theoretic interpretation of classical logic --- proof of the completeness of the P-tS of CL using the techniques used for IPL. Simultaneously, we give a proof of soundness and completeness for first-order intuitionistic logic (IL).
\end{abstract}

\maketitle

\section{Introduction} \label{sec:introduction}

Traditionally, logic has been developed around the idea that the meaning of linguistic terms in logic (e.g., predication, the logical signs, and so on) is determined by reference to external entities. We call this view \textit{denotationalism}. A standard account of it within symbolic logic is \emph{model-theoretic semantics} (M-tS) in which a logical language is interpreted within abstract algebraic structures $\mathfrak{M}$ called `models'. Tarski~\cite{Tarski1936,Tarski2002} formalized a standard definition of logical consequence within this framework: a proposition $\phi$ follows model-theoretically from a context $\Gamma$ if and only if every model of $\Gamma$ is also a model of $\phi$:
\[
  \Gamma \models \phi \quad \text{iff} \quad \text{for all models } \mathcal{M}, \text{ if } \mathcal{M} \models \psi \text{ for all } \psi \in \Gamma, \text{ then } \mathcal{M} \models \phi.
\]
Observe that within this conception of logic, \emph{truth} is conceptually \textit{prior} to inference.

Arguably, this approach to semantics leaves something to be desired. Prawitz~\cite{Prawitz1971ideas} says that explaining the meaning of sentences purely in terms of the conditions under which they are `true' does little to clarify how speakers actually understand them. This is a problem as it seems a reasonable requirement of the semantics of a sentence to do so.  Building on this, Dummett~\cite{Dummett1991logical} says succinctly, ``A theory of meaning must explain what it is to understand an expression; and to understand an expression is to be capable of using it correctly.'' He argues that if meaning were merely a matter of reference, we would have no way to explain how speakers learn the meanings of expressions without already possessing a grasp of the entities to which they refer. 

What alternative is there to the denotational or model-theoretic approach? To begin, Prawitz~\cite{Prawitz2007} has argued that \emph{inference} ought be conceptually \textit{prior} to truth. This sets up the idea that one can ground meaning in terms of inference rather than truth. An idea again succinctly summarized by Dummett~\cite{Dummett1991logical}, ``A grasp of the meaning of a logical constant is manifested in a mastery of the inference rules governing it.'' The broader philosophical view that meaning is based on inferential roles has been dubbed \emph{inferentialism} by Brandom~\cite{Brandom2000}. This may be viewed as a particular instantiation of the \emph{meaning-as-use} paradigm of Wittgenstein~\cite{wittgenstein2009philosophical}, in which `use' is given by inferential role.

This inferential conception of meaning is actually quite standard in how we think about it outside textbooks on symbolic logic. Suppose someone were to ask you for the meaning of the proposition, \textit{``Tammy is a vixen."} Stating that it means that the object denoted by \textit{``Tammy"} belongs to the category of objects denoted by \textit{``vixen"} would be unhelpful.  Rather, we may explain that it means \textit{``Tammy is female"} and \textit{``Tammy is a fox"}. Pushed for clarification, we might explain that the statements stand in the following inferential relationships:
\[
\begin{array}{c}
\infer{\text{Tammy is a vixen}}{\text{Tammy is a fox} & \text{Tammy is female}} \qquad
\infer{\text{Tammy is female}}{\text{Tammy is a vixen}} \qquad 
\infer{\text{Tammy is a fox}}{\text{Tammy is a vixen}}
\end{array}
\]
In so doing, we have given the meaning of the expression ``Tammy is a vixen." 

This approach to meaning may be used within a logical language too. Rather than M-tS, we have what Schroeder-Heister~\cite{SEP-PtS} has called \emph{proof-theoretic semantics} (P-tS). It is the approach to meaning based on `proof' rather than `truth'. Of course, `proof' here does not merely constructions within formal systems, but is understood generally as an object  denoting collections of inferences. This shift and its technical details has substantial and subtle mathematical and conceptual consequences. 

There is a rich literature on P-tS. We defer to Schroeder-Heister~\cite{SEP-PtS,Schroeder2006validity,Schroeder2007modelvsproof} for its broad history. An important point is the following striking, albeit somewhat enigmatic, remarks by Gentzen~\cite{Gentzen}:
\begin{quotation}
    ``The introductions represent, as it were, the ‘definitions’ of the symbols concerned, and the eliminations are no more, in the final analysis, than the consequences of these definitions. This fact may be expressed as follows: In eliminating a symbol, we may use the formula with whose terminal symbol we are dealing only `in the sense afforded it by the introduction of that symbol.'''
\end{quotation}
While he does not explicitly state how introduction rules provide meaning to logical connectives, the foundational intuition for logical inferentialism are unmistakably present. They also suggest where to begin in developing such a semantic theory.  
Consider the inference rules governing conjunction ($\land$):
\[
\infer{\phi \land \psi}{\phi & \psi} \qquad
\infer{\phi}{\phi \land \psi} \qquad 
\infer{\psi}{\phi \land \psi}
\]
They are analogous to the propositions concerning Tammy above, justifying the `and' used there. These rules provide meaning of ``$\phi \land \psi$'' the same way as the rules above provided the meaning of ``Tammy is a vixen''. Of course, this does not yet explain the origin of the elimination rules.

There are several branches of research within P-tS. Much of it has followed an account given by Prawitz~\cite{Prawitz2006natural,Prawitz1973towards} (cf. Schroeder-Heister~\cite{Schroeder2006validity}) based on his normalization results for natural deduction. We elide the details as they are not directly relevant to the work in this paper. Here we are concerned with a related formalism (cf. Gheorghiu and Pym~\cite{PTV-BES}) called \emph{base-extension semantics} (B-eS).

In B-eS, one seeks to give the semantics of logical connectives by defining a \emph{support} relation representing a semantic judgment (analogous to the role of satisfaction in M-tS). It differs from M-tS in that one does not evaluate validity relative to models but rather relative to collections of inference rules over atoms known as \emph{bases}. These rules may include, for example, the following kinds of figures:
\[
\infer{~~C~~}{} \qquad \infer{C}{P_1 & \ldots & P_n} \qquad \infer{C}{\deduce{P_1}{[\mathbb{P}_1]} & \ldots & \deduce{P_n}{[\mathbb{P}_n]}}
\]
A formula $\phi$ follows from a context $\Gamma$ if and only if every $\base{B}$ supporting $\Gamma$ also supports $\phi$; that is,
\[
\Gamma \supp \phi \quad \text{iff} \quad \text{for all bases } \base{B}, \text{ if } \supp_{\base{B}} \psi \text{ for all } \psi \in \Gamma, \text{ then } \supp_{\base{B}} \phi.
\]
where $C$, $P_1,...,P_n$ are atoms and $\mathbb{P}_1,...,\mathbb{P}_n$ are finite sets of atoms. 
Schroeder-Heister~\cite{Schroeder2012categorical} has observed that both M-tS and B-eS adhere to the \emph{standard dogma} of semantics, in which consequence is understood as the transmission of some categorical notion, and he has been a critic of this view. Though B-eS is closely related to M-tS, especially possible world semantics in the sense of Beth~\cite{Beth1955} and Kripke~\cite{kripke1965semantical} (cf. Goldfarb~\cite{goldfarb2016dummett}, Makinson~\cite{makinson2014inferential}, and Stafford and Nascimento~\cite{stafford2023,nascimentothesis}), it remains subtle.  

Sandqvist~\cite{Sandqvist2005inferentialist,Sandqvist2009CL} studied B-eS as a way to present classical logic (CL) without prior commitment to \emph{bivalence}. By `bivalence' we mean that there are exactly two truth values, \emph{True} and \emph{False}. The critique of the denotationalist conception of meaning above renders such a commitment dubious. Nonetheless, classical logic has merit in that, if nothing else, it appears to capture the logical foundations of mathematics rather well. Therefore,  Dummett~\cite{Dummett1978} writes:
\begin{quotation}
    ``In the resolution of the conflict between [the view that generally accepted classical modes of inference ought to be theoretically accommodated, and the demand that any such accommodation be achieved without recourse to bivalence] lies, as I see it, one of the most fundamental and intractable problems in the theory of meaning; indeed, in all philosophy.''
\end{quotation}
This puts Sandqvist's result in perspective: it is a major achievement in both philosophy and mathematical logic. One technical requirement for soundness and completeness is that bases contain rules of the form
\[
\infer{~~C~~}{} \qquad \text{and} \qquad  \infer{~~C~~}{P_1 & \ldots P_n}
\]
Including rules with discharge results in incompleteness. 
% This is a major achievement in programme of P-tS. 

Makinson~\cite{makinson2014inferential} remarks that Sandqvist's proof of completeness is ``not very helpful'' for understanding the system. Accordingly, he has provided an alternative proof of completeness for classical \emph{propositional} logic (CPL) by directly relating the proposed B-eS to the logic's existing truth-functional semantics --- note that he reverses the terms ``soundness'' and ``completeness'' for reasons he discusses. Eckhardt and Pym~\cite{Eckhardt,Eckhardt2} have used this approach to systematically develop the B-eS for normal modal logic.  While Makinson's proof clarifies the classicality of the B-eS for CL, it does so by committing to the model-theoretic interpretation of the logic. In so doing, it forfeits genuinely exposing the inferential nature of the semantics. Rather than demonstrating how inferential reasoning alone recovers classical logic, it sidesteps the issue. Therefore, we seek a completeness proof that also illustrates inferential foundations of the framework. 

Sandqvist~\cite{Sandqvist2015IL} has also provided a B-eS for \emph{intuitionistic propositional logic} (IPL). It follows the semantics for CL for the connectives $\land, \to, \bot$ but takes the following clause for disjunction:
\[
\begin{array}{lcl}
\supp_{\base{B}} \phi \lor \psi & \mbox{iff} & \mbox{for any base $\base{C} \supseteq \base{B}$ and any atom $p$,}  \\ & & \mbox{if $\phi \supp_{\base{C}} p$ and $\psi \supp_{\base{C}} p$, then $\supp_{\base{C}} p$}
\end{array}
\]
For completeness, the bases must range over sets that include rules with discharge,
\[
\infer{C}{\deduce{P_1}{[\mathbb{P}_1]} & \ldots & \deduce{P_n}{[\mathbb{P}_n]}}
\]

The results by Piecha et al~\cite{Piecha2015failure,Piecha2016completeness,Piecha2019incompleteness} show that intuitionistic logic is incomplete for the same semantics but with a more typical clause for disjunction ---- viz. $\supp_{\base{B}} \phi \lor \psi$ iff $\supp_{\base{B}} \phi$ or $\supp_{\base{B}} \psi$). Stafford~\cite{Stafford2021} later showed that the this is a semantics for variants of Kriesel-Putnam logic. 

The clause for disjunction can be justified in several ways. Dummett~\cite{Dummett1991logical} appears to have been the principal inspiration. The clause also recalls Prawitz’s~\cite{Prawitz2006natural} second-order definition of logical signs. A third perspective, followed by Gheorghiu et al.~\cite{gheorghiu2024proof,IMLL,ggp2024practice}, suggests that it arises from treating \emph{definitional reflection} --- in the sense of Halln\"as and Schroeder-Heister~\cite{hallnas1990proof,hallnas1991proof} --- as a closure condition on introduction rules that determines meaning. Of course, definitional reflection can be used either with introduction or with eliminations in the role of definiens --- see de Campos Sanz and Piecha~\cite{sanz2009inversion}. Whatever the case, we note the restriction to quantification over atomic propositions. This remains to be justified philosophically, though the clause does extend to full formulae as expected:
\[
\begin{array}{lcl} \supp_{\base{B}} \phi \lor \psi & \mbox{iff} & \mbox{for any base $\base{C} \supseteq \base{B}$ and any formula $\chi$,} \\ & & \mbox{if $\phi \supp_{\base{C}} \chi$ and $\psi \supp_{\base{C}} \chi$, then $\supp_{\base{C}} \chi$.} \end{array}
\]
However, starting with this formulation would not yield an \emph{inductive} definition of support.

What is most notable about the B-eS of intuitionistic logic is the proof of completeness. Sandqvist~\cite{Sandqvist2015IL} remarks:
\begin{quote}
``The mathematical resources required for the purpose are quite elementary; there will be no need to invoke canonical models, K\"onig’s lemma, or even bar induction. The proof proceeds, instead, by \emph{simulating} an intuitionistic deduction using basic sentences within a base specifically tailored to the inference at hand.''
\end{quote}
Specifically, it proceeds by simulating a $\mathsf{NJ}$-derivation, where $\mathsf{NJ}$ is the natural deduction system for intuitionistic logic given by Gentzen~\cite{Gentzen}. We provide more details below but note that Gheorghiu et al.~\cite{IMLL,BI,NAF} have shown that this approach is quite versatile in the context of constructive logics. Importantly, this approach to completeness clearly reveals how the inferential content of the semantics relates to proof-theoretic expression. 

In this paper, we present a modular such completeness proof for both classical and intuitionistic first-order logic. In the first case, such a proof is lacking in the literature, so our work fills a gap. In the second case, we extend existing work to the first-order setting which demands simulating variables inside bases. The actual semantic clauses are the same for both logics; the only difference lies in the notion of base.  The modularity of this approach suggests that much greater emphasis should be placed on how the notion of base, which is assumed prior to the definition of logical signs, influences their meaning. We regard this as a crucial and unresolved problem within P-tS. 

To deliver this work requires overcoming a couple of challenges. Firstly, to `simulate' $\mathsf{NJ}$-derivations, Sandqvist~\cite{Sandqvist2015IL} makes essential use of rules with discharge. Since such rules are not permitted in the B-eS for classical logic, we require some way of simulating classical deductions, including those with implications, without using discharge. To do this, rather than simulate derivations in the sense of Gentzen~\cite{Gentzen}, we simulate derivations in the sense of Hilbert (see Kleene~\cite{Kleen1952,kleene1967}). This broadens the scope of what can be achieved with a simulation-style proof of completeness. Secondly, the move to first-order logic requires careful treatment of quantified variables in the simulation. We solve this problem by introducing eigenvariables in the simulation and carefully tracking their behaviour. 
\smallskip

\noindent \textbf{Roadmap.} In Section~\ref{sec:fol} we give a terse proof-theoretic background to first-order classical and intuitionistic logic to fix some notations. We proceed in Section~\ref{sec:bes} to give the formal definition of their B-eS. In Section~\ref{sec:soundness} and Section~\ref{sec:completeness} we prove soundness and completeness, respectively. The paper ends in Section~\ref{sec:conclusion} with some discussion about the results presented. 

\smallskip

\noindent \textbf{Notations.} We briefly outline some of the notations used throughout the paper for quick reference; the details are given below. 

\noindent Conventions:
\begin{itemize}
\item $x,y,z, \ldots$ denote variables, 
\item $t,s,r, \ldots$ denote terms
\item $P, Q, \ldots$ denote either predicates or atomic formula
\item $\phi, \psi, \chi$ denote formulae, 
\item $\atset{P}, \atset{Q}, \atset{R}, \ldots$ denote (possibly empty, possibly infinite) sets of atoms
\item $\Gamma, \Delta, \Xi, \ldots$ denote (possibly empty, possibly infinite) sets of formulae, and 
\item $\base{A}, \base{B}, \base{C}, \ldots$ denote atomic systems 
\end{itemize}
The sets of \emph{all} variables, terms, atoms, and formulae are $\setVar$, $\setTerm$, $\setAtoms$, and $\setFormulas$, respectively; the subsets containing only closed elements (where appropriate) are denoted $\setClosedTerm$, $\setClosedAtom$, $\setClosedFormulas$, respectively.

We write $\FV$ to denote the function that takes a term, atom, or formula and returns the set of its free variables. We write $[x \mapsto t]$ to denote a substitution function that replaces all free occurrences of $x$ by $t$. 

We use $\bot, \land, \lor, \to, \forall, \exists$ as the logical signs. While $\bot$ is sometimes taken as an atomic formula, especially in the first-order setting, we reserve that terminology for formulas of the form $P(t_1, \ldots, t_n)$, where $P$ is a predicate symbol and $t_1, \ldots, t_n$ are terms. Intuitively, $\bot$ has some meaning that is expressed in terms of such atomic formulae, rendering it semantically complex.

\section{Background: First-order Classical and Intuitionistic Logic}\label{sec:fol}

In this section, to keep the paper self-contained, we provide a concise definition of both classical and intuitionistic first-order logic, ensuring that key terms and notation are clear. We assume familiarity with these logics and defer to van Dalen~\cite{van1994logic} for further details.

We assume the following are fixed denumerable sets of symbols:
\begin{itemize}[label=--]
    \item $\setVar$  ---  variables $x, y, z, \ldots$
    \item $\setCons$  ---  constants $a, b, c, \ldots$
    \item $\setFunc_n$  ---  function symbols of arity $n$, $f, g, h, \ldots$ 
    \item $\setPred_n$  ---  predicate symbols of arity $n$, $P, Q, R, \ldots$
\end{itemize}
Relative to these sets, we develop the notions of \emph{terms}, \emph{atomic formulas}, and \emph{well-formed formulas} (wffs) as usual. The set of all wffs is denoted $\setFormulas$. We may write $\neg \phi$ to abbreviate $\phi \to \bot$. Again, $\bot$ is not considered an `atomic formula' for it has logical content.

Given a formula $\phi$, its subformulae are defined inductively as follows:
\begin{itemize}
    \item $\phi$ is a subformula of $\phi$
    \item if $\phi=\phi_1\circ\phi_2$ for any $\circ = \{\land,\lor,\to\}$, then the subformulae of $\phi_1$ and $\phi_2$ are subformulae of $\phi$ 
    \item if $\phi = \forall x\psi$ or $\phi = \exists x \psi$, then the subformulae of $\psi[x \mapsto t]$, for every $t \in \setClosedTerm$, are subformulae of $\phi$.
\end{itemize}
Observe that only closed formulae are subformulae. 

We define classical and intuitionistic \emph{consequence} via a Frege-Hilbert axiomatization --- see Troelstra and Schwichtenberg~\cite{Troelstra2000basic}, Kleene~\cite{kleene1967}, and Church~\cite{Church1956} for extensive historical notes. To this end, we require substitutions. We write $[x \mapsto t]$ to denote a \emph{term substitution}; that is, $t'[x \mapsto t]$ denotes the result of uniformly replacing every occurrence of $x$ in $t'$ with $t$. It is clear that the result remains a term. This extends to formulas as expected, where \emph{bound} variables are not substituted. Carefully distinguishing free and bound variables is essential for a precise formulation of first-order logic. Therefore, let $\FV:\setFormulas \to \powerset(\setVar)$ map a formula $\phi$ to the set of its free variables $\FV(\phi)$. A term, atom, or formula is \emph{closed} if it has no free variables. The set of closed terms is denoted $\setClosedTerm$, the set of closed atoms $\setClosedAtom$, and the set of closed formulas $\setClosedFormulas$.

To define \emph{classical} and \emph{intuitionistic} consequence we introduce the concept of an \emph{axiomatization}. Fix a set $\mathbb{X}$ of \emph{formula-variables}. From formula-variables, one constructs \emph{formula-schemes} according to the  grammar
\[
\mathfrak{s}, \mathfrak{t} ::= X \in \mathbb{X} \mid \bot \mid \mathfrak{s} \land \mathfrak{t} \mid \mathfrak{s} \lor \mathfrak{t}  \mid \mathfrak{s} \to \mathfrak{t}  \mid \forall x \mathfrak{s} \mid \exists x \mathfrak{s} \mid \mathfrak{s}_x^t
\]
where here $x$ ranges over variables $(\setVar)$ and $t$ over terms $(\setTerm)$. 

\begin{example}
   Let $X,Y,Z$ be formula-variables. The string $X \to (Y \to X)$ is a formula scheme. This is distinct from a formula $\phi \to (\psi \to \phi)$, where $\phi, \psi, \chi$ are formulas. The string $\bot \to (\bot \to \bot)$ is both a formula and a formula scheme.
\end{example}

As the name suggests, a formula scheme describes the structure of formulas. However, to use it, we require \emph{instantiation}, whereby the formula-variables are systematically and completely replaced with formulas. An \emph{instantiation} is a function $\iota:\mathbb{X} \to \setFormulas$. It extends to formula-schemes as follows:
\[
\iota (\mathfrak{s}) := 
\begin{cases}
    \iota(\mathfrak{s}) & \mbox{if } \mathfrak{s} \in \mathbb{X} \\
    \bot & \mbox{if } \mathfrak{s} = \bot \\
   \iota (\mathfrak{s}_1) \circ   \iota (\mathfrak{s}_2)  & \mbox{if } \mathfrak{s} = \mathfrak{s}_1 \circ \mathfrak{s}_2, \quad \circ \in \{\land, \lor, \to\} \\
    Q x \iota(\mathfrak{s})  & \mbox{if } \mathfrak{s} = Q x \mathfrak{s}, \quad Q \in \{\forall, \exists\}  \\
       \iota(\mathfrak{s})[x \mapsto t]  & \mbox{if } \mathfrak{s} = \mathfrak{s}_x^t  \\
\end{cases}
\]
\begin{example}
    Let $\iota$ be such that $X \mapsto P(x)\land Q(t)$, where $P$ and $Q$ are unary predicates, $x$ is a variable, and $t$ is a term. Then:
    \[
    \iota(\forall x X \to X_x^t) =\forall x(P(x) \land Q(t)) \to (P(t) \land Q(t)).
    \]
\end{example}

Given a set of fomula-schemes $\mathsf{A}$, we define consequence $\Gamma \proves_{\mathsf{A}} \phi$ to mean that there is a valid argument from the axioms $\mathsf{A}$ and assumptions $\Gamma$ that ends with $\phi$. 

\begin{definition}[Consequence from an Axiomatization] \label{def:consequence}
   Let $\mathsf{A}$ be an axiomatization. The $\mathsf{A}$-consequence relation $\proves_\mathsf{A}$ is defined inductively as follows: 
     \begin{itemize}[label=--]
       \item \textsc{Axiom}. If $\mathfrak{s} \in \mathsf{A}$ and $\iota$ is an instantiation, then $\Gamma \proves_\mathsf{A} \iota(\mathfrak{s})$.
       \item \textsc{Hypothesis}. If $\phi \in \Gamma$, then $\Gamma \proves_\mathsf{A} \phi$.
       \item \textsc{Modus Ponens}. If $\Gamma \proves_\mathsf{A} \phi$ and $\Gamma \proves_\mathsf{A} \phi \to \psi$, then $\Gamma \proves_\mathsf{A} \psi$.
       \item \textsc{Generalization}. If $\Gamma \proves_\mathsf{A} \psi \to \phi$ and $x \not \in \FV(\psi)$, then $\Gamma \proves_\mathsf{A} \psi \to \forall x\phi$.
       \item \textsc{Existential Instantiation}. If $\Gamma \proves_\mathsf{A} \phi \to \psi$ and $x \not \in \FV(\psi)$, then $\Gamma \proves_\mathsf{A} \exists x \phi \to \psi$.
    \end{itemize}
\end{definition}

There are many axiomatizations of classical and intuitionistic logic in the literature. The following appears in Kleene~\cite{Kleen1952}:

\begin{definition}[Axioms of First-order Logic]
We define two sets of axioms:
\begin{itemize}
    \item $\setClassical$ comprises all the formula-schemes in Figure~\ref{fig:fol}.
    \item $\setIntuitionistic$ comprises all the formula-schemes in Figure~\ref{fig:fol}, except $\dne$.
\end{itemize} 
\end{definition}

\begin{figure}[t]
\hrule
\vspace{2mm}

\[
\begin{array}{ l @{\hspace{-1mm}}  r }
      X \to (Y \to X) & \hspace{0.4\linewidth}  (\textsc{K}) \\
  (X \to (Y \to Z)) \to \big((X \to Y) \to (X \to Z)\big) & (\textsc{S}) \\
     \forall x X \to X_x^t  & 
   (\ern \forall) \\[1mm]
 X \to (Y \to (X \land Y)) & (\irn{\land}) \\
 X \land Y \to X & (\ern \land_1) \\
 X \land Y \to Y & (\ern \land_2) \\
 X \to X \lor Y & (\irn \lor_1) \\
 Y \to  X \lor Y & (\irn \lor_2) \\
 (X \to Z) \to ((Y \to Z) \to  (X \lor Y \to Z)) & (\ern \lor) \\  
   X_x^t \to \exists x X  & 
   (\irn \exists) \\
   (X \to Y) \to \big((X \to \neg Y) \to \neg X \big) & (\irn \neg)\\
  (X \to \bot) \to (X \to Y) & (\efq) \\
   \hspace{-2ex} \dotfill & \dotfill \\[1mm]
    \neg \neg X \to X & (\dne)
\end{array}
\]

\vspace{2mm}
\hrule
\caption{Axiomatization of First-order Logic ($\setIntuitionistic$ \& $\setClassical$)} \label{fig:fol}
\end{figure}

We do not require much meta-theory about provability in first-order classical and intuitionistic logic. However, the following results will simplify later proofs:

\begin{proposition}[Deduction Theorem] \label{prop:deduction-theorem}
    If $\phi, \Gamma \proves \psi$, then $\Gamma \proves \phi \to \psi$.
\end{proposition}

This is a well-known result with a complex history --- see Franks~\cite{curtis} for a summary. The following are immediate corollaries:

\begin{proposition}\label{prop:or-elim}
    If $\Gamma \proves_{\rn{I}} \phi \lor \psi$, $\phi, \Gamma \proves_{\rn{I}} \chi $, and $\psi, \Gamma \proves_{\rn{I}} \chi $, then  $\Gamma \proves_{\rn{I}}  \chi$.
\end{proposition}

\begin{proposition}\label{prop:efq}
    Let $\rn{A} \in \{\rn{C}, \rn{I}\}$. If $\Gamma \proves_{\rn{A}} \bot$, then $\Gamma \proves_{\rn{A}} \phi$ for any $\phi \in \setFormulas$.
\end{proposition}

\begin{proposition}\label{prop:exists-elim}
    If $\Gamma \proves_{\rn{I}} \exists x \phi$ and $\phi[x\mapsto t], \Gamma \proves_{\rn{I}} \chi$, where $t\in\setClosedTerm$ does not occur elsewhere in $\Gamma$, $\phi$, or $\chi$, then  $\Gamma \proves_{\rn{I}}  \chi$ .
\end{proposition}

This completes the summary of (intuitionistic and classical) first-order logic.

\section{Base-extension Semantics} \label{sec:bes}

\subsection{Atomic System}

We begin with the notion of an \emph{atomic system}. It may be thought of as an agent's commitments regarding the inferential connexions between some atomic sentences. These are pre-logical in the sense that they do not contain any logical signs. This is why $\bot$ is not `atomic' in this paper. For example, Aristotle may believe the fact
\[
\infer{\text{Socrates is human}}{}
\]
and the inference 
\[
\infer{\text{Socrates is mortal}}{\text{Socrates is human}}
\]
What we mean by inference here is that support of the assertion of the premisses gives support for the assertion of the conclusion. For example, Aristotle believes that the assertion ``Socrates is mortal" is supported, even though it does not appear explicitly in the position, for it is deduced from his inferential commitments. 

A position in this sense is expressed mathematically as an \emph{atomic system}, defined as a collection of \emph{atomic rules} that represent facts and inferential relationships. The commitments relative to an atomic system are captured through a derivability judgment.

\begin{definition}[Atomic Rule] \label{def:atomic-rule}
    An atomic rule is a pair $\{\atset{P}_1 \Rightarrow P_1, \ldots \atset{P}_n \Rightarrow P_n\} \Rightarrow P$ in which $\atset{P}_1, \ldots,\atset{P}_1 \subseteq \setClosedAtom$ and $P_1,...,P_n \in \setClosedAtom$. An atomic rule is zero-level if $n=0$, it is first-level if $n \neq 0$ but $\atset{P}_1,\ldots,\atset{P}_n=\emptyset$, and it is second-level otherwise. 
\end{definition} 

We may drop braces and other notation for readability. 

\begin{definition}[Atomic System] \label{def:atomic-system}
    An atomic system $\base{S}$ is a set of atomic rules. 
\end{definition}

\begin{definition}[Derivability in an Atomic System] \label{def:der-base}  Let $\base{S}$ be an atomic system. Derivability in $\base{S}$ is defined inductively as follows:
    \begin{itemize}
        \item \textsc{ref}. If $P \in \atset{P}$, then $\atset{P} \proves_{\base{S}} P$
        \item \textsc{app}. If $\{\atset{P}_1 \Rightarrow P_1, \ldots \atset{P}_n \Rightarrow P_n\} \Rightarrow P \in \base{S}$ and $\atset{P}, \atset{P}_i \proves_{\base{S}} P_i$ for $i=1,\ldots,n$, then $\atset{P} \proves_{\base{S}} P$.
    \end{itemize}
\end{definition}
Thus, Aristotle's position above may be represented symbolically by an atomic system $\base{A}$ containing the rules
\[
\Rightarrow H(s) \qquad \text{and} \qquad  H(s) \Rightarrow M(s)
\]
Moreover, based on this belief set, Aristotle is committed to believing $M(s)$, since we can see that
$\proves_{\base{A}} M(s)$.

This notation suggests a relationship to natural deduction in the sense of Gentzen~\cite{Gentzen}:
\[
\{\atset{P}_1 \Rightarrow P_1, \ldots \atset{P}_n \Rightarrow P_n\} \Rightarrow P \qquad \text{may be expressed as} \qquad \infer{P}{\deduce{P_1}{[\atset{P}_1]} & \ldots & \deduce{P_1}{[\atset{P}_n]}}
\]
That this reading is correct follows from Definition~\ref{def:der-base}. However, since there is no substitution in the application of these rules, they are perhaps more closely related to hereditary Harrop formulae than natural deduction rule figures --- see Gheorghiu and Pym~\cite{NAF} for further details.

Working in FOL means that we have a choice about whether to include variables  in the notion of an atomic system. This would allow us to naturally encode a position such as whoever is human is also mortal,
\[
\infer{\text{$x$ is mortal}}{\text{$x$ is human}}
\]
However, we intend for positions to capture the inferential relationships between complete thoughts, so we restrict ourselves to atomic \emph{sentences} --- that is, closed atoms. Thus, we prefer to represent the above position by including every instance of
\[
\infer{\text{$t$ is mortal}}{\text{$t$ is human}}
\]
as $t$ ranges over all names. Mathematically, either approach is possible, but we believe this choice better aligns with our intended interpretation. It is easy to see that the two formalisms have the same expressive power.

\begin{definition}[Open Atomic Rule] \label{def:open-base}
     An open atomic rule is a pair $\{\atset{P}_1 \Rightarrow P_1, \ldots \atset{P}_n \Rightarrow P_n\} \Rightarrow P$ in which $\atset{P}_1, \ldots,\atset{P}_1 \subseteq \setAtoms$ and $P_1,...,P_n \in \setAtoms$.
\end{definition} 

The nomenclature is slightly off since an atomic rule containing \textit{no variables} is still considered an \emph{open} atomic rule under this definition.  If we were to engage with them more seriously, we would choose a better term (perhaps begin with these as atomic systems and refer to the objects defined in Definition~\ref{def:atomic-rule} as \emph{closed} atomic systems), but this one will suit the short discussion herein. 

The notions of zeroth-, first- and second-level open atomic rule is analogous to the case in Definition~\ref{def:atomic-rule}. The same goes for the notion of an open atomic system (cf. Definition~\ref{def:atomic-system}) and derivability in an open atomic system (cf. Definition~\ref{def:der-base}).  We write  $\vdash_{\base{A}}' A$ to denote that a (possibly open) atom $A$ is derived in the open atomic system $\base{A}$. 

Given an open atomic system $\base{A}$, we may write $\bar{\base{A}}$ to denote its \textit{closure}. That is, let a \textit{total unifier} be an injection $\theta:\setVar \to \setClosedTerm$ that acts as a term substation in the sense given as syntactic structure $S$ (e.g., a term, an atom, a formula, or a set of formulae, etc), we may write $P\theta$ to denote the result of replacing all free variables by their image under $\theta$.  Then, given a rule $r \in \base{A}$, its closure is all its set of instances:
\[
\bar{r} := \{ \atset{P}_1\theta \Rightarrow P_1\theta, \ldots \atset{P}_n\theta \Rightarrow P_n\theta\} \Rightarrow P\theta \mid \text{total unifiers $\theta$} \}
\]
The closure of $\base{A}$ is then formally the union of all these sets:
\[
\bar{\base{A}} := \bigcup_{r \in \base{A}} \bar{r}
\]

Under this closure operation, we see that open atomic systems actually do not add any expressiveness:

\begin{proposition} 
For any open atomic system $\base{A}$ and any atom $P$, there is a total unifer $\theta$ such that
\[
\vdash_{\base{A}}' P \qquad \mbox{iff} \qquad \supp_{\bar{\base{A}}} P\theta
\]
\end{proposition}
\begin{proof}
   By definition of $\bar{\base{A}}$, applying $\theta$ to any sequence of steps (i.e., unfolding of its inductive definition) witnessing $\vdash_{\base{A}}' P$ yields a sequence of steps witnessing $\supp_{\bar{\base{A}}} \bar{P}\theta$. 
    
    Conversely, let $\theta^{-1}$ be the left-inverse of $\theta$. Applying $\theta^{-1}$ to any sequence of steps witnessing $\supp_{\bar{\base{A}}} \bar{P}\theta$ yields a sequence of steps $\vdash_{\base{A}}' P$. 
\end{proof}

\subsection{Support}

We use atomic systems to ground the semantic judgment of \emph{support} ($\supp$) that defines B-eS. We may not necessarily accept all forms of atomic systems. Different conceptions of atomic systems correspond to different logics.  Accordingly, we introduce the idea of a \emph{basis}, which specifies the kinds of atomic systems we work with. Given a fixed basis $\basis{B}$, its elements are called \emph{bases} $\base{B}$. Once a basis is fixed, we always work relative to its bases. Therefore we introduce the notion of \emph{base-extension}, a restricted version of superset (or ``extension") of an atomic system that respects the given basis.

\begin{definition}[Basis]
    A basis $\basis{B}$ is a set of atomic systems.
\end{definition}

\begin{definition}[Base-extension]
    Given a basis $\basis{B}$, base-extension is the least relation satisfying the following:
    \[
    \base{Y} \basesup{\basis{B}} \base{X} \qquad \mbox{iff} \qquad \base{X}, \base{Y} \in \basis{B} \mbox{ and } \base{Y} \supseteq \base{X}
    \]
\end{definition}

Our semantics for the logical constants will be paramterized by the notion of basis. In particular, we consider the following, where $\base{A}$ is an arbitrary atomic systems, 
\begin{itemize}
\item $\basis{C} := \{\base{A} \mid \mbox{$\base{A}$ is zero- or first-level} \}$  
\item $\basis{I} := \{\base{A} \mid \mbox{$\base{A}$ is zero-, first- or second-level} \}$
\end{itemize}
Of course, $\basis{C}$ will correspond to CL and $\basis{I}$ to IL. We understand this correspondence mathematically, but as to why it should be this way remains opaque.

The meaning of logical signs is given by clauses relative to a basis that collectively define a semantic judgment called \emph{support}.

\begin{definition}[Support] \label{def:supp}
    Let $\basis{B}$ be a basis. Support is the smallest relation $\supp$ defined by the clauses of Figure~\ref{fig:support} in which: $\base{B} \in \basis{B}$, all formulae are closed, $\Delta$ is a non-empty set of closed formulae, and $\Gamma$ is a finite (possibly empty) set of formulae. 
\end{definition}

\begin{figure}[t]
\hrule 
\vspace{2mm}
 \[
        \begin{array}{l@{\quad}c@{\quad}l@{\quad}r}
            \supp_{\base{B}} P  & \mbox{iff} &   \proves_{\base{B}} P & \mbox{(At)}  \\[1mm]
            \supp_{\base{B}} \phi \to \psi & \mbox{iff} & \phi \supp_{\base{B}} \psi & (\to) \\[1mm]
            \supp_{\base{B}} \phi \land \psi \hspace{3mm} & \mbox{iff} &    \supp_{\base{B}} \phi \text{ and  } \supp_{\base{B}} \psi& (\land) \\[1mm]
             \supp_{\base{B}} \forall x \phi & \mbox{iff} & \mbox{$\supp_{\base{B}} \phi[x \mapsto t]$ for any $t \in \setClosedTerm$}  & (\forall) \\[1mm]
                   \supp_{\base{B}} \bot & \mbox{iff} &    \supp_{\base{B}} P \text{ for any } P \in \setClosedAtom & (\bot) \\[1mm]
            \supp_{\base{B}} \phi \lor \psi & \mbox{iff} & \mbox{for any $\base{C} \basesup{\mathfrak{B}}\base{B}$ and $P \in \setClosedAtom$,} & \\ & & \mbox{if $\phi \supp_{\base{C}} P$ and $\psi \supp_{\base{C}} P$, then $\supp_{\base{C}} P$}  & (\lor) \\[1mm]
  \supp_{\base{B}} \exists x \phi & \mbox{iff} &  \mbox{for any $\base{C} \basesup{\mathfrak{B}} \base{B}$ and $P \in \setClosedAtom$,} & \\ & & \mbox{if $\phi[x \mapsto t] \supp_{\base{C}} P$ for any $t \in \setClosedTerm$, then $\supp_{\base{C}} P$} & (\exists)  \\[1mm]
   \hspace{-4mm} \Delta \supp_{\base{B}} \phi & \mbox{iff} & \mbox{$\forall \base{C} \basesup{\mathfrak{B}} \baseB$, if $\supp_{\base{C}} \psi$ for any $\psi \in \Delta$, then $\supp_{\base{C}} \phi$ } &  \mbox{(Inf)} \\[1mm]
              \hspace{-1em} \Gamma \supp \phi & \mbox{iff} & \mbox{$\Gamma \supp_{\base{B}} \phi$ for any $\base{B} \in \basis{B}$}
        \end{array}
        \]
    \hrule
    \caption{Base-extension Semantics for First-order Logic} 
    \label{fig:support}
\end{figure}

This definition merits some remarks. Firstly, it is an inductive definition, but the induction measure is not the size of the formula $\phi$. Instead, it is a measure of the ``logical weight" $w$ of $\phi$. Extending Sandqvist~\cite{Sandqvist2015IL},
\[
w(\phi) := 
\begin{cases}
    0 & \mbox{if $\phi \in \setAtoms$} \\
    1 & \mbox{if $\phi = \bot$} \\
   w(\phi_1)+w(\phi_2)+1 & \mbox{if $\phi = \phi_1 \circ \phi_2$ for any $\circ \in \{\to, \land, \lor\}$} \\
   w(\psi) + 1 & \mbox{if $\phi = Qx\psi$ for any $Q \in \{\forall, \exists\}$ and $x \in \setVar$}
\end{cases}
\]
In each clauses of Figure~\ref{fig:support}, the sum of the weights of the complex formulas
flanking the support judgment in the definiendum exceeds the corresponding number for any occurrence of the judgment in the definiens. We refer to induction relative to this measure as \emph{semantic} induction to distinguish it from \emph{structural} induction.

Secondly, the base-extension $\base{C} \supseteq \base{B}$ appears to recall the change of world condition familiar from Kripke's semantics for intuitionistic logic~\cite{kripke1965semantical}. Its origins here comes from an intuitive reading of implication that does have a intuitionistic quality to it. Intuitively, we expect a $\base{B}$ to support an implication $\phi \to \psi$ if whenever it supports $\base{B}$ then also supports $\psi$. However, such a condition would be satisfied vacuously if $\base{B}$ did not support $\phi$. Therefore, we consider arbitrary extension $\base{C} \supseteq \base{B}$ in which $\phi$ is supported and check that under such commitments $\psi$ is also supported.

Thirdly, this semantics assumes a \emph{static} ontology. The M-tS for intuitionistic logic introduced by Kripke~\cite{kripke1965semantical} is based on a quasi-ordered set of models whose domains may expand along the ordering. That is, it has an ontology that grows as knowledge increases. Grzegorczyk~\cite{grzegorczyk1964interpretation} proposed an alternative with a constant domain, where the ontology remains fixed across all stages. But G\"ornemann~\cite{gornemann1971logic} (see also Klemke~\cite{klemke1971henkin} and Gabbay~\cite{gabbay1969montague}) demonstrated that this yields an intermediate logic extending intuitionistic logic $\mathsf{I}$ with the additional axiom:
\[
\forall x(\mathfrak{s} \lor \mathfrak{t}) \to (\mathfrak{s} \lor \forall x\mathfrak{t}),
\]
where $x$ is not free in $\mathfrak{s}$. By contrast, in B-eS there is no need for the domain to change with the base. This makes things comparatively simple. 

We assign semantics to sentences --- that is, closed formulae --- based on the view that meaning is a property of complete assertions. Nonetheless, we may extend support to all well-formed formulas (wffs) by considering the universal closure of formulas with free variables. That is, let $\phi$ and let $\FV(\phi) = \{x_1,\ldots,x_n\}$, 
\[
\uniclose{\phi} := \forall x_1,\ldots,x_n \phi
\]
To make this deterministic, we may choose that $x_1,\ldots,x_n$ are the free variables in $\phi$ as it is read right to left. Having established such a closure, we define the support of open formulae as follows:
\[
\supp_{\base{B}} \phi \qquad \mbox{iff} \qquad \supp_{\baseB}  \uniclose{\phi} \tag*{(\textsc{wff})} \label{cl:wff}
\]

To conclude this section, we observe some important results about semantics. They are proved by Sandqvist~\cite{Sandqvist2015IL}. The following hold with respect to either  $\basis{C}$ or $\basis{J}$ --- we write $\basis{B}$ as a parameter. 

\begin{proposition}[Monotonicity] \label{prop:monotonicity}
    If $\supp_{\base{B}} \phi$ and $\base{C} \basesup{\mathfrak{B}} \base{B}$, then $\supp_{\base{C}} \phi$.
\end{proposition}

\begin{proposition} \label{prop:CPL:atomic-cut}
    For any $\atset{Q}, \atset{P} \subseteq \setClosedAtom$ (finite) and $P \in \setClosedAtom$ and base $\base{B}$,
    \[
    \atset{Q}, \atset{P} \proves_{\base{B}} P \qquad \mbox{iff} \qquad \text{$\forall \baseX \basesup{\mathfrak{B}} \baseB$, if $\proves_{\baseX} Q$ for $Q \in \atset{Q}$, then $\atset{P}\proves_{\baseX} P$}
    \]
\end{proposition}

This completes the presentation of the B-eS for FOL. It remains to prove soundness and completeness for classical and intuitionistic logic.

\section{Soundness} \label{sec:soundness}

The soundness of an axiomatization $\mathsf{A}$ with respect to a semantics is the property that only valid consequences may be derived --- that is, if $\Gamma \vdash_{\mathsf{A}} \phi$, then $\Gamma \supp \phi$.
We follow the standard approach and show that the semantics judgment $(\supp)$ satisfies all the axioms and all the rules of derivation.

\begin{lemma}[Axiom] \label{lem:axiom}
    If $A \in \setClassical$ (resp. $A \in \setIntuitionistic$) and $\basis{B}=\basis{C}$ (resp. $\basis{B} = \basis{I}$), then $\supp \sigma A$ for any substitution $\sigma$. 
\end{lemma}

\begin{proof}
    We proceed by case analysis on the axioms in $\setClassical$ or $\setIntuitionistic$. For the $A \in \setClassical \cap \setIntuitionistic$, we leave $\basis{B}$ ambiguous as both $\basis{C}$ and $\basis{J}$ apply. We use $\phi$, $\psi$, and $\chi$ as arbitrary instantiations of formula-variables. \smallskip

    \noindent $\basis{B} \in \{\basis{C}, \basis{I}\}$: 
    \begin{itemize} 
        \item (K). Let $\base{C} \basesup{\mathfrak{B}} \base{B}$ be arbitrary such that the following hold:
        \begin{enumerate}[label=(\roman*),leftmargin=2cm] 
       \item $\supp_{\base{B}} \phi$
       \item $\supp_{\base{C}} \psi$
        \end{enumerate}
        Since $\base{C}$ is arbitrary, it follows by (Inf) from (i) and (ii) that $\supp_{\base{B}} \psi \to \phi$. By (i), we also obtain $\supp_{\base{B}} \phi \to (\psi \to \phi)$, as required.

        \item (S). Let $\base{D} \basesup{\mathfrak{B}} \base{C} \basesup{\mathfrak{B}} \base{B}$ be arbitrary and assume:
\begin{enumerate}[label=(\roman*),leftmargin=2cm] 
    \item $\supp_{\base{B}} (\phi \to (\psi \to \chi))$,
    \item $\supp_{\base{C}} (\phi \to \psi)$,
    \item $\supp_{\base{D}} \phi$.
\end{enumerate}
By (ii) and (iii) using (Inf), we obtain $\supp_{\base{D}} \psi$. Applying (i), we then have $\supp_{\base{C}} (\phi \to \chi)$. Hence, we conclude:
\[
\supp_{\base{B}} ((\phi \to \psi) \to (\phi \to \chi)).
\]
Thus, it follows that:
\[
\supp \phi \to (\psi \to \chi) \to \big( (\phi \to \psi) \to (\phi \to \chi) \big),
\]
as required.

        \item ($\irn \forall$). Given a term $t$, define $[t]$ as the set of all its closures:
\[
[t] := \{t[x_1 \mapsto s_1]\ldots[x_n \mapsto s_n] \mid \{x_1,\ldots,x_n\} = \FV(t), \quad s_1,\ldots,s_n \in \setClosedTerm \}.
\]
Let $\base{B}$ be arbitrary such that $\supp_{\base{B}} \forall x \phi$. By $(\forall)$, we have $\supp_{\base{B}} \phi[x \mapsto s]$ for any $s \in \setClosedAtom$. Since $[t] \subseteq \setClosedAtom \subseteq $ for any term $t$, it follows \emph{a fortiori} that $\supp_{\base{B}} \phi[x \mapsto s]$ for all $s \in [t]$. Thus, by \ref{cl:wff} and $(\forall)$, we conclude that $\supp_{\base{B}} \phi[x \mapsto t]$ for any $t \in \setTerm$. Consequently, we derive:
\[
\supp \phi \to \phi[x \mapsto t], \quad \text{for any } t \in \setTerm,
\]
as required.

    \item ($\irn \land$). Let $\base{B} \in \basis{B}$ and $\base{C} \basesup{\basis{B}} \base{B}$ be such that $\supp_{\base{C}} \phi$ and  $\supp_{\base{C}} \phi$. By ($\land$), $\supp_{\base{C}} \phi \land \psi$. Hence, by (Inf) and ($\to$) twice, $\supp \phi \to (\psi \to \phi \land \psi)$, as required. 
    \item ($\ern{\land}_1$). Immediate by (Inf), ($\to$) and ($\land$).
    \item ($\ern{\land}_2$). Immediate by (Inf), ($\to$) and ($\land$).
    \item ($\irn{\lor}_1$). Let $\base{B} \in \basis{B}$ be such that $\supp_{\base{B}} \phi$. Let $\base{C} \basesup{\basis{B}} \base{B}$ and $P$ be arbitrary such that $\phi \supp_{\base{C}} P$ and $\psi \supp_{\base{C}} P$. By Proposition~\ref{prop:monotonicity}, $\supp_{\base{C}} \phi$. Hence, by (Inf),  $\supp_{\base{B}} P$. Whence, by ($\lor$), $\supp_{\base{B}} \phi \lor \psi$. Finally, $\supp \phi \to \phi \lor \psi$ by (Inf) since $\base{B}$ was arbitrary.
    \item ($\irn{\lor}_2$). \emph{Mutatis mutandis} on the preceding case. 
    \item ($\ern \lor$). Let $\base{D} \basesup{\basis{B}} \base{C} \basesup{\basis{B}} \base{B}$ be arbitrary such that
    \begin{enumerate}[label=(\roman*)]
        \item $\supp_{\base{B}} \phi \to \chi$
        \item $\supp_{\base{C}} \psi \to \chi$
        \item $\supp_{\base{D}} \phi \lor \chi$
    \end{enumerate}
    We desire to show 
     \begin{enumerate}
        \item[(iv)] $\supp_{\base{D}} \chi$
    \end{enumerate}
     Observe that by Proposition~\ref{prop:monotonicity}, from (i) and (ii) we have:
    \begin{enumerate}
        \item[(i$'$)] $\supp_{\base{D}} \phi \to \chi$
        \item[(ii$'$)] $\supp_{\base{D}} \psi \to \chi$
        \end{enumerate}
    We proceed by semantic induction on $\chi$:
    \begin{itemize}
       \item $\chi \in \setClosedAtom$. Immediate by ($\lor$) and (Inf) on (i$'$) and (ii$'$).
       \item $\chi = \bot$. By ($\bot$), it suffices to show $\supp_{\base{D}} P$ for any $P \in \setClosedAtom$. From (i$'$) and (ii$'$) is easy to see $\phi \supp_{\base{D}} P$ and $\psi \supp_{\base{D}} P$. The desired result obtains by ($\lor$) and (Inf) on these statements.
       \item $\chi = \chi_1 \land \chi_2$.  By ($\land$), it suffices to show $\supp_{\base{D}} \chi_1$ and  $\supp_{\base{D}} \chi_2$. They obtain immediately by the \emph{induction hypothesis} (IH). 
       \item $\chi = \chi_1 \lor \chi_2$. Let $\base{E} \basesup{\basis{B}} \base{B}$ and $P \in \setClosedAtom$ be arbitrary be such that $\chi_1 \supp_{\base{E}} P$ and $\chi_2 \supp_{\base{E}} P$. It follows from (i$'$) and (ii$'$) that $\phi \supp_{\base{E}} P$ and $\psi \supp_{\base{E}} P$ using (Inf), ($\to$) and ($\lor$). By Proposition~\ref{prop:monotonicity} we have $\supp_{\base{E}} \phi \lor \psi$. Hence, by ($\lor$), we have $\supp_{\base{E}} P$. Whence, since $\base{E} \basesup{\basis{B}} \base{D}$, we have $\supp_{\base{D}} \chi$, as required. 
       \item $\chi = \chi_1 \to \chi_2$. Let $\base{E} \basesup{\basis{B}} \base{D}$ be such that $\supp_{\base{E}}\chi_1$. Observe $\phi \supp_{\base{E}} \chi_2$ and $\psi \supp_{\base{E}} \chi_2$ by (Inf) and ($\to$) on (i$'$) and (ii$'$). By Proposition~\ref{prop:monotonicity} on (iii), $\supp_{\base{E}} \phi \lor \psi$. Hence, by the IH, $\supp_{\base{E}} \chi_2$. Whence, $\supp_{\base{D}} \chi_1 \to \chi_2$ by (Inf) and ($\to$) since $\base{E} \basesup{\basis{B}} \base{B}$ was arbitrary. 
       \item $\chi = \forall x \chi'$.  By ($\land$), it suffices to show $\supp_{\base{D}} \chi'[x \mapsto t]$ for any $t \in \setClosedTerm$. This is immediate by the (IH). 
       \item $\chi = \exists x \chi'$. Let $\base{E} \basesup{\basis{B} \base{D}}$ and $P$ be arbitrary such that $\supp_{\base{E}}\chi'[x \mapsto t]$ for any $t\in \setClosedAtom$. It follows by (Inf), ($\to$), and $(\exists)$ on (i$'$) and (ii$'$) that $\phi \supp_{\base{E}} P$ and $\psi \supp_{\base{E}} P$. Hence, by ($\lor$), from (iii), infer $\supp_{\base{E}} P$. Whence, $\supp_{\base{D}} \chi$ by ($\exists$) as $\base{E} \basesup{\basis{B}} \base{D}$ was arbitrary. 
   \end{itemize}
   This completes the induction. It now follows by (Inf) and ($\to$) using (i), (ii), and (iii) that $\supp (\phi \to \chi)\to\big((\psi \to \chi) \to (\phi \lor \psi \to \chi)\big)$, as required. 
    
    \item ($\irn \exists$). Let $\base{B} \in \basis{B}$ be such that $\supp_{\base{B}} \phi[x \mapsto t]$. Let $\base{C} \basesup{\basis{B}} \base{B}$ and $P \in \setClosedAtom$ be arbitrary such that $\phi[x \mapsto s]\supp_{\base{C}} P$ for any $s \in \setClosedTerm$. By Proposition~\ref{prop:monotonicity}, $\supp_{\base{C}} P$. Hence, by (Inf), $\supp_{\base{B}} \exists x \phi$. Whence, $\supp \phi[x \mapsto t] \to \exists x \phi$, as required. 
    \item ($\irn \neg$). Let $\base{D} \basesup{\basis{B}} \base{C} \basesup{\basis{B}} \base{B}$ be arbitrary such that
    \begin{enumerate}[label=(\roman*)]
        \item $\supp_{\base{B}} \phi \to \psi$
        \item $\supp_{\base{C}} \phi \to \neg \psi$
        \item $\supp_{\base{D}} \phi$.  
    \end{enumerate}
    We desire to show
    \begin{itemize}
        \item[(iv)] $\supp_{\base{D}} \bot$ 
    \end{itemize}
    Observe by ($\to$) and (Inf) on (i) and (ii) using (iii) that $\supp_{\base{D}} \psi$ and $\supp_{\base{D}} \neg \psi$. It follows by (Inf) that $\supp_{\base{D}} \bot$. Hence, $\supp (\phi \to \psi) \to \big((\phi \to \neg \psi) \to \neg \phi\big)$, as required.
    \item ($\efq$).  Let $\base{C} \basesup{\basis{B}} \base{B}$ be arbitrary such that
    \begin{enumerate}[label=(\roman*)]
        \item $\supp_{\base{B}} \phi \to \bot$
        \item $\supp_{\base{C}} \phi$. 
    \end{enumerate}
    We desire to show
    \begin{enumerate}
        \item[(iii)]  $\supp_{\base{C}} \psi$
    \end{enumerate}
    Observe by (i) and (ii), $\supp_{\base{C}} \bot$. We proceed by semantic induction on $\psi$. 
        \begin{itemize}
            \item $\psi = P \in \setClosedAtom$. Immediate by  $(\bot)$.
            \item $\psi = \bot$. Trivial by the observation. 
            \item $\psi =\psi_1 \land \psi_2$. By ($\land$), it suffices to show $\supp_{\base{C}} \phi_1$ and $\supp_{\base{C}} \phi_2$. These obtains immediately from the \emph{induction hypothesis} (IH).
            \item $\psi = \psi_1 \lor \psi_2$. Let $\base{D} \basesup{\basis{B}} \base{C}$ and $P$ be arbitrary such that $\psi_1 \supp_{\base{D}} P$ and  $\psi_2 \supp_{\base{D}} P$. By Proposition~\ref{prop:monotonicity}, $\supp_{\base{D}} \bot$. Thus, by ($\bot$),  $\supp_{\base{D}} P$. Hence, by ($\lor$), $\supp_{\base{C}} \psi_1 \lor \psi_2$ as $\base{D} \basesup{\base{C}}$ was arbitrary.  
            \item $\phi = \phi_1 \to \phi_2$. Let $\base{C} \basesup{\mathfrak{B}} \base{B}$ be arbitrary such that $\supp_{\base{C}} \phi_1$. Since $\supp_\base{B} \bot$, infer $\supp_{\base{C}} \bot$ (by Proposition~\ref{prop:monotonicity}). By the IH, infer   $\supp_{\base{C}} \phi_2$. The desired result follows from (Inf) since $\base{B}$ was arbitrary.
            \item $\phi =  \forall x \psi$. By the IH, $\supp_{\base{C}} \phi[x \mapsto t]$ for any $t \in \setClosedTerm$. The desired results follows from $(\forall)$.
            \item $\psi = \exists x \psi'$. Let $\base{D} \basesup{\basis{B}} \base{C}$ and $P$ be arbitrary such that $\phi[x \mapsto t] \supp_{\base{D}} P$ for any $t \in \setClosedAtom$. By Proposition~\ref{prop:monotonicity}, $\supp_{\base{D}} \bot$. Thus, by ($\bot$),  $\supp_{\base{D}} P$. Hence, by ($\lor$), $\supp_{\base{C}} \exists x \psi$ as $\base{D} \basesup{\base{C}}$ was arbitrary.  
        \end{itemize}
        This completes the induction.
\end{itemize}

\noindent $\basis{B} = \basis{C}$.
        \begin{itemize}
        \item ($\dne$). This case is given in Sandqvist~\cite{Sandqvist2009CL} --- the proof proceeds by induction on $\phi$ with the base case given by a contradiction that makes us of the assumption that bases are direct. 
        \end{itemize}
        This completes the case analysis.
\end{proof}

In the following three lemmas, $\basis{B}$ may be any set of atomic systems.  

\begin{lemma}[Hypothesis]
    If $\phi \in \Gamma$, then $\Gamma \supp \phi$.
\end{lemma}
\begin{proof}
    Immediate by (Inf).
\end{proof}

\begin{lemma}[Modus Ponens]
    If $\Gamma \supp \phi$ and $\Gamma \supp \phi \to \psi$, then $\Gamma \supp \psi$.
\end{lemma}
\begin{proof}
    Let $\base{B}$ be such that $\supp_{\base{B}} \gamma$ for $\gamma \in \Gamma$. By (Inf), it suffices to show $\supp_\base{B} \psi$. 
    
    By (Inf) on the assumption of $\base{B}$, we obtain (1) $\supp_{\base{B}} \phi$ and (2) $\supp_\base{B} \phi \to \psi$. Hence, by (Inf), on (2) using (1), we obtain $\supp_{\base{B}} \psi$, as required.  
\end{proof}

\begin{lemma}[Generalization] \label{lem:generalization}
    If $\Gamma \supp \psi \to \phi$ and $x \not \in \FV(\psi)$, then $\Gamma \supp \psi \to \forall x \phi$. 
\end{lemma}
\begin{proof}
    Let $\base{B}$ be such that $\supp_{\base{B}} \gamma$ for  $\gamma \in \Gamma, \psi$. Hence, by (Inf) and ($\to$), $\supp_{\base{B}} \phi$. If $x \in \FV(\phi)$, the result obtains by \ref{cl:wff} and $(\forall)$. Otherwise, $x \not \in \FV(\phi)$ so $\phi[x \mapsto t] = \phi$ for any $t \in \setClosedTerm$ and the desired result obtains by ($\forall$).
\end{proof}

\begin{lemma}[Existential Instantiation]
    If $\Gamma \supp \phi \to \psi$ and $x \not \in \FV(\psi)$, then $\Gamma \supp \exists x\phi \to \psi$.
\end{lemma}
\begin{proof}
    Let $\base{B} \in \basis{B}$ be such that $\supp_{\base{B}} \gamma$ for $\gamma \in \Gamma$. 
    Let $\base{C} \basesup{\basis{B}} \base{B}$ be arbitrary such that
    \begin{enumerate}[label=(\roman*)]
    \item
    $\supp_{\base{C}} \exists x \phi$
    \end{enumerate}
    Observe that by the assumption,
       \begin{enumerate}
        \item[(ii)]  $\phi[x \mapsto t] \supp_{\base{B}} \psi$ for any $t \in \setClosedTerm$
    \end{enumerate}
    We desire to show
    \begin{enumerate}
        \item[(iii)] $\supp_{\base{C}} \psi$
    \end{enumerate}
   To this end, we proceed by semantic induction on $\psi$:
   \begin{itemize}
       \item $\psi \in \setClosedAtom$. Immediate by ($\exists$).
       \item $\psi = \bot$. By ($\bot$), it suffices to show $\supp_{\base{C}} P$ for any $P \in \setClosedAtom$. This obtains by $(\exists)$ using Proposition~\ref{prop:monotonicity}. 
       \item $\psi = \psi_1 \land \psi_2$. By ($\land$), it suffices to show $\supp_{\base{C}} \psi_1$ and $\supp_{\base{C}} \psi_2$. These obtains immediately from the  \emph{induction hypothesis} (IH).
       \item $\psi = \psi_1 \lor \psi_2$. Let $\base{D} \basesup{\basis{B}} \base{C}$ and $P$ be arbitrary such that $\psi_1 \supp_{\base{D}} P$ and $\psi_2 \supp_{\base{D}} P$. By (ii), $\phi[x \mapsto t] \supp_{\base{D}} P$ for any $t \in \setClosedAtom$. Hence, by (i), $\supp_{\base{D}} P$. Whence, $\supp_{\base{C}} \psi_1 \lor \psi_2$, by ($\lor$). 
       \item $\psi = \psi_1 \to \psi_2$. Let $\base{D} \basesup{\basis{B}} \base{C}$ be arbitrary such that $\supp_{\base{D}} \psi_1$. It follows from the IH that $\supp_{\base{D}} \psi_2$. Hence, by (Inf) and ($\to$), $\supp_{\base{C}} \psi_1 \to \psi_2$, as required.  
       \item $\psi = \forall x \psi'$. By ($\forall$), it suffices to show $\supp_{\base{C}} \psi[x \mapsto t]$ for any $t \in \setClosedTerm$. This obtains immediately from the IH. 
       \item $\psi = \exists x \psi'$. Let $\base{D} \basesup{\basis{B}} \base{C}$ and $P$ be arbitrary such that $\psi'[x \mapsto t] \supp_{\base{D}} P$ for any $t \in \setClosedTerm$.  From (ii), it follows that $\phi[x \mapsto t] \supp_{\base{D}} P$ for any $t \in \setClosedTerm$. Thus, by (i), $\supp_{\base{D}} P$. Hence, $\supp_{\base{C}} \exists x \psi'$ since $\base{D} \basesup{\basis{B}} \base{C}$ was arbitrary. 
   \end{itemize}
   This completes the induction.
\end{proof}

These lemmas collectively demonstrate soundness of classical consequence with respect to this semantics:

\begin{theorem}[Soundness] \label{thm:soundness}
       If $\proves_{\rn{A}} \phi$ with $\rn{A}=\rn{C}$ (resp. $\rn{A} = \rn{I}$), then $\supp \phi$ over the basis $\basis{B}=\basis{C}$ (resp. $\basis{B}=\basis{I}$).
\end{theorem}

\begin{proof}
 Lemma~\ref{lem:axiom} -- Lemma~\ref{lem:generalization} (below) exactly correspond to the clauses in Definition~\ref{def:consequence}. Hence, support ($\supp$) subsumes classical (resp. intuitionistic) consequence, as required. 
\end{proof}

\section{Completeness} \label{sec:completeness}

The completeness of an axiomatization $\mathsf{A}$ with respect to a semantics is the property that all valid consequences may be derived --- that is, if $\Gamma \supp \phi$, then $\Gamma \vdash_{\mathsf{A}} \phi$. To show this for the B-eS in this paper, we employ the strategy used by Sandqvist~\cite{Sandqvist2015IL} for intuitionistic propositional logic, which we describe below.

The key idea underlying this simulation technique is the systematic translation of formulas into atomic sentences. Suppose $\supp \phi$ and that $\phi$ contains a subformula $a \land b$, where $a, b \in \mathbb{B}$. Following the natural deduction rules governing conjunction, Sandqvist makes sure to include the following rules in the ``specifically tailored'' base $\base{N}$ for $\phi$:
\[
\infer{r}{~~a & b~~} \qquad \infer{~~a~~}{r} \qquad \infer{~~b~~}{r}
\]
where $r$ is a fresh atomic sentence representing $a \land b$. This means that $r$ behaves in $\base{N}$ as $a \land b$ behaves in $\mathsf{NJ}$~\cite{Gentzen} --- that is, as in intuitionistic propositional logic. 

More generally, each subformula $\chi$ of $\phi$ is assigned a corresponding basic counterpart $\kflat{\chi}$ --- for example, $r = \kflat{(a \land b)}$. The major work is establishing their equivalence within $\base{N}$:
\[
\chi \supp_{\base{N}} \kflat{\chi} \quad \text{and} \quad \kflat{\chi} \supp_{\base{N}} \chi.
\]
Since we assume $\supp \phi$, it follows that $\supp \kflat{\phi}$, and given that every rule in $\base{N}$ corresponds to an intuitionistic natural deduction rule, we conclude $\vdash \phi$, as required. 

Applying this technique to classical logic and the present setup requires some work. There are two principal challenges. Firstly, simulating Gentzen's $\rn{NK}$, the classical analogue of $\rn{NJ}$, demands second-level rules of implication introduction,
\[
\infer{\phi \to \psi}{\deduce{\psi}{[\phi]}}
\]
However, rules with discharge are not permitted in $\basis{C}$ as admitting them comes at the cost of soundness --- for example, $\dne$ would no longer be valid.  Secondly, in this paper we are working in the first-order setting which means the simulations requires accounting for quantification and open formulae, while atomic rules only contain closed atoms. It no longer suffices to consider only sub-formulae of the given valid sequent $\Gamma \seq \phi$, nor can we associate to each such sub-formula $\rho$ a \emph{unique} atom $r$. These challenges notwithstanding, the strategy goes through as before. By being careful about the setup, we can overcome both of these challenges. We now proceed with the technical details. 

Given a set of formulae $\Gamma$ and a formula $\gamma$, let $\Xi$ be the set of sub-formulae. Recall that if $\phi$ contains $x$ as a free variables, the subformulae of $\forall x \phi$ and $\exists x \phi$ are the sentences $\phi[x \mapsto t]$ for $t \in \setClosedTerm$. Fix an injection $\kflat{(-)}: \Xi \to \setClosedAtom$.

% Fix a set of formulae $\Gamma$ and a formula $\gamma$:
% \begin{itemize}
% \item Let $\mathcal{X}$ be the set of variables and predicates that appear in $\Gamma$ or $\gamma$, together with \emph{all} constants and functions
%     \item Let $\Xi$ be set of closed formulae comprising only elements in $\mathcal{X}$
%     \item Fix an injection $\kflat{(-)}: \Xi \to \setClosedAtom$ 
% \end{itemize}
% Intuitively, $\Xi$ captures all the kinds of formulae that we may encounter as $\Gamma$ and $\gamma$ unfold in the semantics. 

To handle the quantifiers and free variables, we introduce \emph{eigenvariables}. They are a set of constants that we use as variables during the simulation:
\begin{itemize}
\item Let $\setEigen \subseteq \setCons$ be the set of constants that do not appear in $\Gamma$ or $\phi$
    \item Let $\alpha:\setVar \to \setEigen$ be a bijection. 
\item Extend $\kflat{(-)}$ to open formulae using $\alpha$ such that an open formulae is simulated by the atom that simulates its universal closure. That is, if $\{x_1,\ldots,x_m\} = \FV(\phi)$, then define 
\[
\bar{\phi} := \phi[x_1 \mapsto \alpha(x_1)],\ldots,[x_m \mapsto \alpha(x_m)]
\]
and assign
\[
\kflat{\phi} :=  \kflat{(\bar{\phi})} \tag{$\textsc{eigen}$} \label{def:eigen}
\]
\end{itemize}
This defines the encoding of formulae as atoms required for the simulation. 

To recover formulae from the `flattening' operation, we require functions  $\ksharpl{(-)}:\setClosedAtom \to \setFormulas$ and  $\ksharpr{(-)}:\setClosedAtom \to \setFormulas$   satisfying some conditions. They both decode by reversing $\kflat{(-)}$ but they behave differently with respect to encodings of open formulae. This is important for relating constructions in the simulation bases $\base{K}$ and $\base{J}$ to construction in $\rn{C}$ and $\rn{I}$, respectively. The conditions are as follows: 
\begin{itemize}
    \item If $P=\kflat{\phi}$ such that $\phi$ is closed and does \emph{not} contain any eigenvariables, then $\ksharpl{P} = \ksharpr{P} = \phi$ 
    \item If $P=\kflat{\phi}$ where  $\phi$ is closed but contains eigenvariables, then $\ksharpl{P}$ and $\ksharpr{P}$ distinguish whether or not we want the closed or open version of the formula, respectively:
   \[
\ksharpl{P} := \uniclose{(\phi[a_1 \mapsto \alpha^{-1}(a_1)],\ldots,[a_n \mapsto \alpha^{-1}(a_n)])} \tag{$\textsc{var}_1$}
\]
    and
    \[
\ksharpr{P} := \phi[a_1 \mapsto \alpha^{-1}(a_1)],\ldots,[a_n \mapsto \alpha^{-1}(a_n)] \tag{$\textsc{var}_2$}
\]
where $\{a_1,\ldots,a_n\}$ are all the eigenvariables in $\phi$
\item  If $P = \kflat{\phi}$ where $\phi$ is open, then
\[
\ksharpl{P} = \uniclose{\phi} \qquad  \mbox{ and } \qquad \ksharpr{P} = \ksharpr{(\bar{\phi})}
\]
\item If $P$ is not in the image of $\kflat{(-)}$, then $\ksharpl{P}$ and $\ksharpr{P}$ can be any formula, but they must be the same formula (i.e., $\ksharpl{P} = \ksharpr{P}$)
\end{itemize}

Following the plan by Sandqvist~\cite{Sandqvist2015IL} outlined above, we now simulate a proof system for FOL. However, rather than simulating a natural deduction system, we simulate the Hilbert-Frege systems in Section~\ref{sec:fol}.

\begin{definition}[Natural Base]
Given $\kflat{(-)}$ for some sequence $\Gamma \seq \gamma$ we define two bases:
\begin{itemize}
    \item The classical natural base  $\base{K}$ is given by all instances of the atomic rules in the \emph{first} section of Figure~\ref{fig:simulation} together with $\kflat{(\dne)}$.
    \item The intuitionistic natural base (relative to  $\kflat{(-)}$) $\base{J}$ is given by all instances of the atomic rules in Figure~\ref{fig:simulation} except $\kflat{(\dne)}$.
\end{itemize}
In both cases, the symbols $\phi, \psi, \chi,\xi$ ranger over $\Xi$, $P$ ranges over $\setClosedAtom$, $x$ ranges over $\setVar$, and $t$ ranges over $\setTerm$, but with the constraint that $x \not \in \FV(\xi)$ in the instances of $(\gen)$.
\end{definition}

\begin{figure}
\hrule
 \[
\begin{array}{ l @{\hspace{-1mm}}  r }
    \Rightarrow \kflat{\big(\phi \to (\psi \to \phi)\big)} & \hspace{0.4\linewidth}  \kflat{(\textsc{K})} \\
 \Rightarrow \kflat{\big((\phi \to (\psi \to \chi)) \to \big((\phi \to \psi) \to (\phi \to \chi)\big) \big)} & \kflat{(\textsc{S})} \\
    \Rightarrow \kflat{\big(\forall x \phi \to \phi[x \mapsto t]\big)}  & 
   \kflat{(\ern \forall)} \\
   \{\Rightarrow \kflat{\phi}, \Rightarrow \kflat{\big(\phi \to \psi\big)} \} \Rightarrow \kflat{\psi} & \kflat{(\mp)} \\
  %\mathcal{S}_{\alpha(x)} \mid 
  \{\Rightarrow \kflat{(\xi \to \phi[x \mapsto \alpha(x)])}\} \Rightarrow \kflat{(\xi \to \forall x\phi)} & 
   \kflat{(\gen)} \\
   \dotfill & \dotfill \\[1mm]
 \Rightarrow \kflat{\big(\phi \to (\psi \to (\phi \land \psi)) \big)}& \kflat{(\irn{\land})} \\
 \Rightarrow \kflat{\big(\phi \land \psi \to \phi\big)} & \kflat{(\ern \land_1)} \\
 \Rightarrow \kflat{\big(\phi \land \psi \to \psi\big)} & \kflat{(\ern \land_2)} \\
 \Rightarrow \kflat{\big(\phi \to \phi \lor \psi\big)} & \kflat{(\irn \lor_1)} \\
 \Rightarrow  \kflat{\big(\psi \to  \phi \lor \psi\big)} & \kflat{(\irn \lor_2)} \\
 \Rightarrow \{\kflat{\phi} \Rightarrow P, \kflat{\psi} \Rightarrow P, \Rightarrow  \kflat{(\phi \lor \psi)} \} \Rightarrow P & \kflat{(\ern \lor)} \\  
   \Rightarrow  \kflat{\big(\phi_x^t \to \exists x \phi\big)}  & 
   \kflat{(\irn \exists)} \\
   \Rightarrow  \kflat{\big((\phi \to \psi) \to \big((\phi \to \neg \psi) \to \neg \phi \big)\big)} & \kflat{(\irn \neg)}\\
  \{  \Rightarrow \kflat{\bot} \} \Rightarrow P & \kflat{(\efq)} \\
  \{\Rightarrow \kflat{(\exists \phi)},  \kflat{ \big(\phi[x\mapsto t]\big)} \Rightarrow P\} \Rightarrow P & \kflat{(\ern \exists)}\\
   \hspace{-2ex} \dotfill & \dotfill \\[1mm]
    \Rightarrow  \kflat{\big(\neg \neg \phi \to \phi)} & \kflat{(\dne)}
    
\end{array}
\]
    \hrule
    \caption{Simulation Bases $\base{K}$ and $\base{J}$} 
    \label{fig:simulation}
\end{figure}

This definition requires some remarks. 
Clearly, $(\kflat{K})$, $(\kflat{S})$, $\kflat{(\ern \forall)}$, $\kflat{(\irn \land)}$, $\kflat{(\irn \land)}$, $\kflat{(\ern{\land}_1)}$, $\kflat{(\ern{\land}_2)}$, $\kflat{(\irn \land)}$, $\kflat{(\irn{\lor}_1)}$, $\kflat{(\irn{\lor}_2)}$, $\kflat{(\ern{\lor})}$, $\kflat{(\ern{\neg})}$, and $(\kflat{\dne})$ are intended to simulate corresponding axioms in Figure~\ref{fig:simulation}. The rules $\kflat{\mp}$ and $\kflat{\gen}$ are intended to simulate \textsc{modus ponens} and \textsc{generalization} in Definition~\ref{def:consequence}. The behaviours $\textsc{hypothesis}$ and $\textsc{axiom}$ follow from $\textsc{ref}$ and $\textsc{app}_2$, respectively, in  Definition~\ref{def:der-base}. We have $\kflat{(\ern \lor)}$ and $\kflat{\efq}$ to take the form of `rules' (as opposed to an `axiom' like the others) because we require the conclusion to be broader than $\Xi$ so that we can simulate the semantic clause for the rule. 

% \begin{example}\label{ex:sim}
%     We show how Example~\ref{ex:gen} is simulated in $\base{K}$ and $\base{J}$. We have the following:
%         \begin{align}
%         &\kflat{\big(\forall x (P(x) \to Q(x))\big)} \tag{\textsc{ref}}\\
%         &\kflat{\big(\forall x P(x)\big)}  \tag{\textsc{ref}}\\
%          &\kflat{\big(\forall x \big(P(x) \to Q(x)) \to  \big(P(a) \to Q(a))\big)} \tag{\textsc{app}$_1$ -- $\kflat{(\ern \forall)}$ } \\
%          &\kflat{\big(\forall x P(x) \to P(a)\big)}  \tag{\textsc{app}$_1$ -- $\kflat{(\ern \forall)}$ } \\
%        &\kflat{\big(P(a) \to Q(a)\big)} \tag{\textsc{app}$_2$ -- $\kflat{(\mp)}$} \\
%         &\kflat{P(a)} \tag{\textsc{app}$_2$ -- $\kflat{(\mp)}$}  \\
%         &\kflat{Q(a)} \tag{\textsc{app}$_2$ -- $\kflat{(\mp)}$} \\
%         &\kflat{(Q(a) \to (\top \to Q(a)))} \tag{\textsc{app}$_1$ -- $\kflat{(K)}$}\\
%          &\kflat{(\top \to Q(a))} \tag{\textsc{app}$_2$ -- $\kflat{(\mp)}$} \\
%         &\kflat{(\top \to \forall x Q(x))} \tag{\textsc{app}$_2$ -- $\kflat{(\gen)}$} \\
%         &\vdots \notag \\
%         &\kflat{\top}  \notag \\
%         &\kflat{(\forall x Q(x))} \tag{\textsc{app}$_2$ -- $\kflat{(\mp)}$} 
%     \end{align}
%     The constant $a=\alpha(x)$ is an \emph{eigenvariable}. Intuitively, it is a constant about which we have no information so it represents an arbitrary individual. 
% \end{example}

Relative to this setup, we require the following lemmas: 
\begin{itemize}
    \item \textbf{AtComp}. \labelandtag{lem:CPL:basic-completeness}{AtComp} Let $\atset{P} \subseteq \setClosedAtom$ and $P \in \setClosedAtom$, and let $\base{B}$ be a base: 
    \[\atset{P} \supp_{\base{B}} P  \qquad \mbox{iff} \qquad \atset{P} \proves_{\base{B}} P
    \]
    \item \textbf{Flat}. \labelandtag{lem:CPL:flat-equivalence}{Flat}  For any $\base{A}' \basesup{\mathfrak{B}} \base{A}$,
    \[
    \supp_{\base{A}'} \kflat{\phi} \qquad \mbox{iff} \qquad  \supp_{\base{A}'} \phi
    \]
    holds when 
    \begin{itemize}
    \item $\base{A} = \base{K}$ and  $\basis{B} = \basis{C}$ if $\phi$ contains only $\bot, \to, \forall$ as logical signs, or 
    \item $\base{A} = \base{J}$ and  $\basis{B} = \basis{I}$ if $\phi$ is any formula.
    \end{itemize}
    \item \textbf{Nat}.  \labelandtag{lem:CPL:sharpening}{Nat} Let $\atset{P} = \kflat{\Delta}$ for some $\Delta \subseteq \setClosedFormulas$   and $P \in \setClosedAtom$:
    \[
    \mbox{if $\atset{P} \proves_{\base{K}} P$, then $\ksharpl{\atset{P}} \vdash_{\rn{C}} \ksharpr{P}$.}
    \]
     and 
    \[
    \mbox{if $\atset{P} \proves_{\base{J}} P$, then $\ksharpl{\atset{P}} \vdash_{\rn{I}} \ksharpr{P}$.}
    \]
\end{itemize}
The proofs of these lemmas are technical and uninformative so deferred to the end of this section.

\begin{theorem}[Completeness] \label{thm:CPL:completeness-support}
    Let $\Gamma$ be finite. If $\Gamma \supp_{} \phi$ with $\basis{B}=\basis{C}$ (resp. $\basis{B}=\basis{I}$), then $\Gamma \vdash_{\setClassical} \phi$ (resp. $\Gamma \vdash_{\setIntuitionistic} \phi$)
\end{theorem}
\begin{proof}
Assume $\Gamma \supp_{} \phi$ with $\basis{B}=\basis{C}$ (resp. $\basis{B}=\basis{I}$).
Without loss of generality by \ref{cl:wff}, $\Gamma$ and $\phi$ comprise closed formulae. Let $\kflat{(-)}$ be a flattening operator and let $\base{A}=\base{K}$ (resp. $\base{A}=\base{J}$ be the associated simulation base. By~\ref{lem:CPL:flat-equivalence}, $\kflat{\Gamma} \supp_{\base{A}} \kflat{\phi}$. Hence, by~\ref{lem:CPL:basic-completeness}, $\kflat{\Gamma} \proves_{\base{A}} \kflat{\phi}$. Whence, by~\ref{lem:CPL:sharpening}, $\ksharpl{(\kflat{\Gamma})} \vdash_{\setClassical} \ksharpr{(\kflat{\phi})}$ (resp. $\ksharpl{(\kflat{\Gamma})} \vdash_{\setIntuitionistic} \ksharpr{(\kflat{\phi})}$). That is, $\Gamma \vdash_{\setClassical} \phi$ (resp. $\Gamma \vdash_{\setIntuitionistic} \phi$), as required. 
\end{proof}

\begin{lemma}[\ref{lem:CPL:basic-completeness}]
    Let $\atset{P} \subseteq \setClosedAtom$ and $P \in \setClosedAtom$, and let $\base{B}$ be a base: 
    \[\atset{P} \supp_{\base{B}} P \qquad \mbox{iff} \qquad \atset{P} \proves_{\base{B}} P
    \]
\end{lemma}
\begin{proof}
        Let us assume that $\atset{P} = \{p_1, \dots, p_n \}$. We reason as follows:
        \begin{align}
      \atset{P} \supp_{\baseB} P \qquad &\mbox{iff} \qquad \text{$\forall \baseB \basesup{\mathfrak{B}} \baseX$, if $\supp_{\baseX} P_1, \dots, \supp_{\baseX} P_n$, then $\supp_{\baseX} P$} \tag{Inf}\\
        & \mbox{iff} \qquad \text{$\forall \baseX \basesup{\mathfrak{B}} \baseB$, if $\proves_{\baseX} P_1, \dots, \proves_{\baseX} P_n$, then $\proves_{\baseX} P$} \tag{At} \\
        \qquad & \mbox{iff} \qquad  \atset{P} \proves_{\baseB} P \tag{Proposition~\ref{prop:CPL:atomic-cut}}
        \end{align}
\end{proof}

% \begin{proposition}\label{prop:sim-mp}
%     For any $\base{A}' \basesup{\mathfrak{B}} \base{A}$, if $\proves_{\base{A}'} \kflat{(\phi \to \psi)}$ and $\proves_{\base{A}'} \kflat{\phi}$, then $\proves_{\base{A}'} \kflat{\psi}$
% \end{proposition}
% \begin{proof}
%     induction on the anteceedent
% \end{proof}

\begin{proposition}\label{prop:CPL:flat-sim-clauses}
    The following hold for any $\base{A}' \basesup{\mathfrak{B}} \base{A}$ where except in the (vii) all formulae are closed:
    \begin{enumerate}[label=(\roman*)]
        \item $\proves_{\base{A}'} \kflat{\bot}$ iff $\proves_{\base{A}'} P$ for any $P \in \setClosedAtom$
        \item $\proves_{\base{A}'} \kflat{(\phi \land \psi)}$ iff $\proves_{\base{A}'} \kflat{\phi}$  and $\proves_{\base{A}'} \kflat{\psi}$ (where $\phi$ and $\psi$ do not contain eigenvariables). 
        \item $\proves_{\base{A}'} \kflat{(\phi \lor \psi)}$ iff, for any $\base{A}'' \basesup{\mathfrak{B}} \base{A}$ and $P \in \setClosedAtom$, if $\phi \proves_{\base{A}''} P$  and $\psi \proves_{\base{A}''} P$, then $\proves_{\base{A}''} P$.
        \item $\proves_{\base{A}'} \kflat{(\phi \to \psi)}$ iff $\kflat{\phi} \proves_{\base{A}'} \kflat{\psi}$ (where $\phi$ does not contain eigenvariables).
        \item $\proves_{\base{A}'} \kflat{(\forall x \phi)}$ iff $ \proves_{\base{A}'} \kflat{(\phi[x \mapsto t])}$ for any $t \in \setClosedTerm$
        \item $\proves_{\base{A}'} \kflat{(\exists x \phi)}$ iff, for any $\base{A}'' \basesup{\mathfrak{B}} \base{A}'$ and $P \in \setClosedAtom$, if $\kflat{(\phi[x \mapsto t])} \proves_{\base{A}'} P$ for any $t \in \setClosedTerm$, then $ \proves_{\base{A}'} P$ (where $\phi$ does not contain eigenvariables).
        \item $\proves_{\base{A}'} \kflat{\chi}$ iff $ \proves_{\base{A}'} \kflat{(\chi[x \mapsto t])}$ for any $t \in \setClosedTerm$.
    \end{enumerate}
\end{proposition}
\begin{proof}
We show each claim separately and divide each claim into its two directions: 
\begin{enumerate}[label=(\roman*)]
\item Absurdity ($\bot$):
   \begin{itemize}
      \item  LHS $\Longrightarrow$ RHS. Since $\proves_{\base{A}'} \kflat{\bot}$, it follows from $\kflat{\efq}$ that  $\proves_{\base{A}'} P$ for any $P \in \setClosedAtom$. 

      \item RHS $\Longrightarrow$ LHS. Since $\proves_{\base{A}'} P$ for any $P \in \setClosedAtom$ and $\kflat{\bot} \in \setClosedAtom$, the desired result obtains \emph{a fortiori}.
      \end{itemize}
      \item Conjunction ($\land$):
       \begin{itemize}
\item  LHS $\Longrightarrow$ RHS. Immediate by $\kflat{\ern{\land}_1}$ and $\kflat{\ern{\land}_2}$ and $\kflat{\mp}$.  
\item  RHS $\Longrightarrow$ LHS. Immediate by $\kflat{\irn{\land}}$ and $\kflat{\mp}$. 
\end{itemize}
      \item Disjunction ($\lor$):  
             \begin{itemize}
\item  LHS $\Longrightarrow$ RHS.           Immediate by $\kflat{\ern{\lor}}$ using Proposition~\ref{prop:monotonicity}.
\item  RHS $\Longrightarrow$ LHS. By $\kflat{(\irn \lor_1)}$ and $\kflat{(\irn \lor_2)}$, observe $\kflat{\phi} \proves_{\base{A}'} \kflat{(\phi \lor \psi)}$ and $\kflat{\phi} \proves_{\base{A}'} \kflat{(\phi \lor \psi)}$. Thus, the  desired result obtains from the assumption by choosing $\base{A}''=\base{A}'$ and $P = \kflat{(\phi \lor \psi)}$.
\end{itemize}
\item Implication ($\to$):
    \begin{itemize}
\item  LHS $\Longrightarrow$ RHS. Let $\base{A}'' \basesup{\mathfrak{B}} \base{A}'$ be such that $\proves_{\base{A}''} \kflat{\phi}$.  
From LHS, it follows that $\proves_{\base{A}''}\kflat{(\phi \to \psi)}$ by Proposition~\ref{prop:monotonicity}. Hence (by $\kflat{\mp}$),  $\proves_{\base{A}''} \kflat{\psi}$. The desired result follows from Proposition~\ref{prop:CPL:atomic-cut} since $\base{A}'' \basesup{\mathfrak{B}} \base{A}'$ was arbitrary.
\item  RHS $\Longrightarrow$ LHS. Follows from standard approaches to the Deduction Theorem for classical logic --- see, for example, Herbrand~\cite{herbrand1930recherches}. 
% \{induction on the length of the antecedent --- deduction theorem https://www.personal.kent.edu/~rmuhamma/Philosophy/Logic/Deduction/4-deductionTheorem.htm } 
\end{itemize}
\item Universal Quantifier ($\forall$):
\begin{itemize} 
\item  LHS $\Longrightarrow$ RHS. Since $\proves_{\base{A}'} \kflat{(\forall x \phi)}$, the desired result follows from $\kflat{\mp}$ and $\kflat{(\ern \forall)}$.  

\item  RHS $\Longrightarrow$ LHS. 
Let $\top$ be an arbitrary formula such that $\proves_{\base{A}'} \kflat{\top}$ --- for example, $\top = \bot \to (\bot \to \bot)$. From RHS, infer $\top \proves_{\base{A}'} \kflat{(\forall x \phi)}$. By (iv), observe $ \proves_{\base{A}'} \kflat{(\top \to \forall x \phi)}$. Let $a \in \setEigen$ be arbitrary. From RHS, $\proves_{\base{A}'}\kflat{(\phi[x \mapsto a])}$ for any eigenvariable $a \in \setEigen$. From this, infer $\kflat{\top} \proves_{\base{A}'}\kflat{(\phi[x \mapsto a])}$ (by Proposition~\ref{prop:CPL:atomic-cut} and Proposition~\ref{prop:monotonicity}). By (iv), we have $\proves_{\base{A}'}\kflat{(\top \to \phi[x \mapsto a])}$. Hence, applying $\irn \forall$, $\proves_{\base{A}'}\kflat{(\top \to \forall x\phi)}$. Whence (by $\kflat{\mp}$), $\proves_{\base{A}'}\kflat{(\forall x\phi)}$, as required.
\end{itemize}
\item Existential Quantifier ($\exists$):
\begin{itemize}
     \item LHS $\Longrightarrow$ RHS. Let $\base{A}'' \basesup{\basis{B}} \base{A}$ be arbitrary such that $\kflat{(\phi[x \mapsto t])} \proves_{\base{A}''} P$ for any $t \in \setClosedTerm$ and $P \in \setClosedAtom$. By Proposition~\ref{prop:monotonicity}, observe $\proves_{\base{A}''} \kflat{(\exists x \phi)}$. It now follows by $\kflat{(\ern \exists)}$ that $\proves_{\base{A}
     ''} P$, as required. 
     % Fix arbitrary $\base{A}'' \basesup{\mathfrak{B}} \base{A}'$ and $P \in \setClosedAtom$ such that $\kflat{(\phi[x \mapsto t])} \proves_{\base{A}'} P$ for any $t \in \setClosedTerm$. It follows by case (iv),  $\proves_{\base{A}'} \kflat{(\phi[x \mapsto t] \to P)}$ \todo{P not in the domain}. By $\kflat{(\ern \exists)}$, $\proves_{\base{A}'} \kflat{(\exists x \phi \to P)}$. Simultaneously, from LHS,  by Proposition~\ref{prop:monotonicity}, $\proves_{\base{A}''} \kflat{(\exists x \phi)}$. From these judgements, observe $\proves_{\base{A}''} P$ by $\mp$.
    \item RHS $\Longrightarrow$ LHS. Observe that $\kflat{(\phi[x \mapsto t])} \proves_{\base{A}'} \kflat{(\exists x \phi)}$ for any $t \in \setClosedAtom$ by $\kflat{(\irn \exists)}$ and $\kflat{\mp}$. The desired result obtains from instantiating RHS with $\base{A}'' = \base{A}'$ and $P = \kflat{(\exists \phi)}$.
\end{itemize}
\item Open:
     \begin{itemize}
\item  LHS $\Longrightarrow$ RHS. Let $\{x_1,\ldots,x_n\} = \FV(\psi)$. By (\ref{def:eigen}), we have from the assumption that $\proves_{\base{A}} \kflat{(\psi[x_1 \mapsto a_1]\ldots[x_n \mapsto a_n])}$, for some $a_1,\ldots,a_n \in \setEigen$. Hence, by $\kflat{\gen}$ applied $n$-times,  $\proves_{\base{A}} \kflat{\uniclose{\psi}}$. Whence, by $\kflat{(\forall)}$ and $\kflat{\mp}$ applied $n$-times, $\proves_{\base{A}} \kflat{\psi[x \mapsto t]}$, as required. 

\item  RHS $\Longrightarrow$ LHS. By (\ref{def:eigen}), $\kflat{\psi} = \kflat{(\psi[x \mapsto a])}$ for some $a \in \setEigen$. Since $a \in \setClosedTerm$, the result obtains \emph{a fortiori}.
\end{itemize}
\end{enumerate}
\end{proof}

% \begin{proposition}[Generalization Simulation]~\label{prop:sim-gen}
%     For any $\base{A}' \basesup{\mathfrak{B}} \base{A}$, if $\proves_{\base{A}'} \kflat{\phi}$, then $\proves_{\base{A}'} \kflat{(\forall x \phi)}$. 
% \end{proposition}
% \begin{proof}
%     \{xxx}
% \end{proof}

\begin{lemma}[\ref{lem:CPL:flat-equivalence}]
     For any $\base{A}' \basesup{\mathfrak{B}} \base{A}$,
    \[
    \supp_{\base{A}'} \kflat{\phi} \qquad \mbox{iff} \qquad  \supp_{\base{A}'} \phi
    \]
    holds when 
    \begin{itemize}
    \item $\base{A} = \base{K}$ and  $\basis{B} = \basis{C}$ if $\phi$ contains only $\bot, \to, \and$, and $\forall$ as logical signs, or 
    \item $\base{A} = \base{J}$ and  $\basis{B} = \basis{I}$ if $\phi$ is any formula.
    \end{itemize}
\end{lemma}
\begin{proof}
    We proceed by semantic induction on $\phi$. In all cases, the result follows from the use of Proposition~\ref{prop:CPL:flat-sim-clauses}. We show the cases for the quantifiers, the others being similar to work by Sandqvist~\cite{Sandqvist2015IL}:
    \begin{itemize}
       %  \item $\phi \in \setClosedAtom \cap \Xi$. Trivial as $\kflat{(-)}$ is the identity on $\setClosedAtom \cap \Xi$.
       %  \item $\phi = \bot$. We reason as follows:
       %  \begin{align}
       %  \supp_{\base{A}'} \kflat{\bot} & \qquad \mbox{iff} \qquad \proves_{\base{A}'} \kflat{\bot} \tag*{(At)} \\
       %  & \qquad \mbox{iff} \qquad  \mbox{$\proves_{\base{A}'} P$ for any $P \in \setClosedAtom$} \tag*{(Prop.~\ref{prop:CPL:flat-sim-clauses})} \\
       %  & \qquad \mbox{iff} \qquad  \mbox{$\supp_{\base{A}'} P$ for any $P \in \setClosedAtom$} \tag*{(IH)} \\
       %  & \qquad \mbox{iff} \qquad  \supp_{\base{A}'} \bot \tag{$\bot$} 
       %  \end{align}
       %  \item $\phi = \phi_1 \land \phi_2$. \todo{xxx}
       %  \item $\phi = \phi_1 \lor \phi_2$. \todo{xxx}
       %  \item $\phi = \phi_1 \to \phi_2$. We reason as follows:
       %  \begin{align}
       %  \supp_{\base{A}'} \kflat{(\phi \to \psi)} & \qquad \mbox{iff} \qquad   \proves_{\base{A}'} \kflat{(\phi \to \psi)} \tag*{(At)} \\
       % & \qquad \mbox{iff} \qquad \kflat{\phi} \proves_{\base{A}'} \kflat{\psi}  \tag*{(Prop.~\ref{prop:CPL:flat-sim-clauses})} \\
       %  & \qquad \mbox{iff} \qquad   \phi \proves_{\base{A}'} \psi  \tag*{(IH)} \\
       %  & \qquad \mbox{iff} \qquad   \proves_{\base{A}'} \phi \to \psi \tag*{($\to$)} 
       %  \end{align}
        \item $\phi = \forall x \psi$. We reason as follows:
       \begin{align}
        \supp_{\base{A}'} \kflat{(\forall x \psi)} & \qquad \mbox{iff} \qquad \proves_{\base{A}'} \kflat{(\forall x \psi)} \tag*{(At)} \\
        & \qquad \mbox{iff} \qquad  
        \mbox{$\proves_{\base{A}'} 
 \kflat{(\psi[x \mapsto t])}$ for any $t \in \setClosedTerm$} \tag*{(Prop.~\ref{prop:CPL:flat-sim-clauses})} \\
         & \qquad \mbox{iff} \qquad   
         \mbox{$\supp_{\base{A}'} (\psi[x \mapsto t])$ for any $t \in \setClosedTerm$}  \tag*{(IH)} \\
        & \qquad \mbox{iff} \qquad   
        \supp_{\base{A}'} 
 \forall x \psi \tag*{($\forall$)} 
        \end{align}
        \item $\phi = \exists x \psi$. We reason as follows:
        \begin{align}
        \supp_{\base{A}'} \kflat{(\exists x \psi)} & \qquad \mbox{iff} \qquad \proves_{\base{A}'} \kflat{(\exists x \psi)} \tag*{(At)} \\
        & \qquad \mbox{iff} \qquad  \mbox{for any $\base{A''} \basesup{\basis{B}} \base{A}$ and $P \in \setClosedAtom$,} \tag*{(Prop.~\ref{prop:CPL:flat-sim-clauses})}\\ 
       & \hspace{17mm} \mbox{if $\kflat{(\psi[x \mapsto t])} \proves_{\base{A}''} P 
$ for any $t \in \setClosedTerm$, then $\proves_{\base{A}''} P$} \notag \\
         & \qquad \mbox{iff} \qquad  \mbox{for any $\base{A''} \basesup{\basis{B}} \base{A}$ and $P \in \setClosedAtom$,} \tag*{(IH)}\\ 
       & \hspace{17mm} \mbox{if $\kflat{(\psi[x \mapsto t])} \supp_{\base{A}''} P 
$ for any $t \in \setClosedTerm$, then $\supp_{\base{A}''} P$} \notag  \\
        & \qquad \mbox{iff} \qquad   
        \supp_{\base{A}'} 
 \exists x \psi \tag*{($\exists$)} 
        \end{align}
 %        \item $\phi$ is open. We reason as follows:
 %        \begin{align}
 %        \supp_{\base{A}'} \kflat{\phi} & \qquad \mbox{iff} \qquad \proves_{\base{A}'} \kflat{\phi} \tag*{(At)} \\
 %        % & \qquad \mbox{iff} \qquad   
 %        % \proves_{\base{A}'} \kflat{\big(\uniclose{\phi}\big)}  \tag*{(Prop.~\ref{prop:flat-vec})} \\
 %        & \qquad \mbox{iff} \qquad   
 %        \mbox{$\proves_{\base{A}'} 
 % \kflat{(\phi[x \mapsto t])}$ for any $t \in \setClosedTerm$} \tag*{(Prop.~\ref{prop:CPL:flat-sim-clauses})} \\
 %         & \qquad \mbox{iff} \qquad   
 %         \mbox{$\supp_{\base{A}'} (\phi[x \mapsto t])$ for any $t \in \setClosedTerm$}  \tag*{(IH)} \\
 %        & \qquad \mbox{iff} \qquad   
 %        \supp_{\base{A}'} 
 % \uniclose{\phi}  \tag*{($\forall$)} \\
 % & \qquad \mbox{iff} \qquad  
 %        \supp_{\base{A}'} 
 % \phi & \tag*{\ref{cl:wff}} 
 %        \end{align}
    \end{itemize}
    This completes the induction.
\end{proof}

\begin{lemma}[\ref{lem:CPL:sharpening}]
Let $\atset{P} = \kflat{\Delta}$ for some $\Delta \subseteq \setClosedFormulas$ not containing eigenvariables,  and $P \in \setClosedAtom$:
    \[
    \mbox{ if $\atset{P} \proves_{\base{K}} P$, then $\ksharpl{\atset{P}} \vdash_{\rn{C}} \ksharpr{P}$.}
    \]
     and 
    \[
    \mbox{ if $\atset{P} \proves_{\base{J}} P$, then $\ksharpl{\atset{P}} \vdash_{\rn{I}} \ksharpr{P}$.}
    \]
\end{lemma}
\begin{proof}
    We proceed by induction on how $\atset{P} \proves_{\base{A}} P$ obtains:
    \begin{itemize}[label=--]
        \item \textsc{ref}. If $P=\kflat{\phi}$ where $\phi$ is a closed formula not containing eigenvariables or $P$ is not in the image of $\kflat{(-)}$, then $\ksharpl{\atset{P}} \proves \ksharpr{P}$ obtains by \textsc{hypothesis} in Definition~\ref{def:consequence}. Otherwise, it obtains by $\ern{\forall}$ and Proposition~\ref{prop:deduction-theorem}.
        \item \textsc{app}$_1$. We have $\Rightarrow P \in \base{K}$ (resp. $\Rightarrow P \in \base{I}$). By construction of $\base{K}$ (resp. $\base{I}$), there is $\mathsf{A} \in \setClassical$ (resp.  $\mathfrak{a} \in \setIntuitionistic$) such that $\ksharpr{P} = \iota(\mathfrak{a})$ for some instantiation $\iota$. Hence, $\ksharpl{\atset{P}} \proves_{\setClassical} \ksharpr{P}$ (resp. $\ksharpl{\atset{P}} \proves_{\setIntuitionistic} \ksharpr{P}$) obtains by \textsc{axiom} in Definition~\ref{def:consequence}. 
        \item  \textsc{app}$_2$.  We have a rule $\rn{r} = \{\atset{P}_1 \Rightarrow P_1, \ldots, \atset{P}_n \Rightarrow P_n\} \Rightarrow P \in \base{A}$, where $\base{A} \in \{\base{K}, \base{J}\}$, such that $\atset{P}, \atset{P}_i \proves_\base{A} P_i$ for $i = 1, \ldots, n$. By the \emph{induction hypothesis} (IH), $\ksharpl{\atset{P}},\ksharpl{\atset{P}}_i \proves_\base{A} \ksharpr{P}_i$ for $i = 1, \ldots, n$. We proceed by case analysis on $\rn{r}$ to show $\ksharpl{\atset{P}} \Rightarrow \ksharpr{P}$:
        \begin{itemize} 
             \item $\rn{r} = \kflat{\mp}$.  We have $i=2$, with $P_1 = \kflat{\phi}$, $P_2 = \kflat{(\phi \to \psi)}$,  and $P = \kflat{\psi}$ for some formula $\phi$ and $\psi$. The result obtains from the IH by \textsc{modus ponens} in Definition~\ref{def:consequence}.
            \item $\rn{r} = \kflat{\gen}$. We have $i=1$ with $\atset{P}_1=\emptyset$, $P_1=\kflat{(\xi \to \phi[x\mapsto \alpha(x)])}$, and $P=\{\kflat{(\xi \to \forall x\phi)}\}$, such that $x \not \in \FV(\xi)$. By (\textsc{var}), $\ksharpr{P}_1 = \xi \to \phi$ and $\ksharpr{P} = \xi \to \forall x \phi$. Also since $\Delta \subseteq \setClosedFormulas$ and does not contain any eigenvariables,  we have $x \not \in \kflat{\Delta}$. The desired result obtains by \textsc{generalization}.
            \item $\rn{r} = \kflat{(\ern \lor)}$. We have $i=3$, with $\atset{P}_1 = \{\kflat{\phi}\}$, $\atset{P}_2 = \{\kflat{\psi}\}$, $\atset{P}_3 = \emptyset$, $P_1=P_2=P$, and $P_3= \kflat{(\phi \lor \psi)}$ for some formulas $\phi, \psi$, and  $P$ is any closed atom. The desired result obtains from the IH using Proposition~\ref{prop:or-elim}.
             \item $\rn{r} = \kflat{(\efq)}$. We have $i=1$, with $\atset{P}_1 = \emptyset$, $P_1 = \kflat{\bot}$, and $P$ is any closed atom. The desired result obtains from the IH using Proposition~\ref{prop:efq}. 
             \item  $\rn{r} = \kflat{(\ern \exists)}$. We have $i =2$, with $\atset{P}_1= \{\kflat{(\exists \phi)}\}$, $\atset{P}_2= \{\kflat{(\phi[x \mapsto t])}\}$ and $P$ is any closed atom. The desired result obtains from the IH using Proposition~\ref{prop:exists-elim}. 
        \end{itemize}
        This completes the case analysis.
    \end{itemize}
    This completes the induction.
\end{proof}

\section{Discussion} \label{sec:conclusion}

We have shown that first-order classical logic is sound and complete for the proof-theoretic semantics given by Sandqvist~\cite{Sandqvist2005inferentialist,Sandqvist2009CL}. Simultaneously, we have extended the result to first-order intuitionistic result. The proofs of soundness and completeness are constructive and do not rely on the extant model-theoretic readings of these logics. 

Curiously, the semantics for first-order classical and intuitionistic logic can be reduced to a difference in the notion of `base'. It is remarkable that a seemingly small change can have such precise and large consequences for which logic is expressed by the support relation. There are various notion of atomic systems in the literature that make a suitable notion of base according to various philosophical positions --- see, for example,  Piecha and Schroeder-Heister~\cite{Piecha2017definitional,Schroeder2016atomic} and Sandqvist~\cite{Sandqvist2015hypothesis}. 

The type of base comes up in showing  that $\dne$ holds --- that is, that $\supp \neg \neg \phi \to \phi$. Of course, $\dne$ may be replaced by some other laws such as Pierce's Law,
\[
((X \to Y) \to X) \to X \tag{$\rn{pl}$}
\]
We may then ask, `What is it about various notions of base that make them intuitively \emph{classical}?' Having answered that, we may substitute `classical' for other logics of interest --- intuitionistic, modal, relevant, linear, and so on.    

For example, following Sandqvist~\cite{sandqvistwld}, an alternate candidate for a 'classical' notion of base would be one containing rules of the form
\[
\infer{C}{P_1 & \ldots & P_m & \deduce{Q_1}{[\atset{Q}_1]} & \ldots & \deduce{Q_n}{[\atset{Q}_n]}}
\]
where $C,P_1,\ldots,P_m,Q_1,\ldots,Q_n \in \setClosedAtom$ and $\atset{Q}_1, \ldots, \atset{Q}_n \subseteq \setClosedAtom$ (finite and possibly empty).  

The observation suggests that the semantic content of the logical signs is not entirely encapsulated by their clauses in the semantics, but also depends on the notion of consequence at the pre-logical level (i.e., derivability in a base). This is also witnessed in the work by Gheorghiu et al.~\cite{IMLL,BI} for substructural logics in which the notion of base was itself made substructural.  

Finally, we would draw attention to the fact that in the proof of completeness, certain signs --- namely, $\bot$, $\lor$, $\exists$ --- required distinctive treatment. While all the other logical connectives were handled essential, as in the definition of classical consequences via a Frege-Hilbert axiomatization, these signs seemingly demanded to be given a proper rule in the simulation base (rather than an `axiom'). On the one hand, this is, perhaps, an accident of the setup and has no real significance. On the other hand, it may say something fundamental about the logical treatment of `choice' and `uncertainty'. Indeed, negation/absurdity is already  well-known to be subtle in proof-theoretic semantics --- see, for example K\"urbis~\cite{kurbis2019proof} and Gheorghiu and Pym~\cite{PTV-BES,NAF}. 

\subsection*{Funding}
This work was supported by the EPSRC grant EP/R006865/1.

\subsection*{Acknowledgements}
We are grateful to Tao Gu and David J. Pym for the many discussions that led to the development of this paper.

\bibliographystyle{abbrv}
\bibliography{bib}

\end{document}